\documentclass[12pt,letterpaper]{amsart}

\usepackage{amsmath,amssymb,amsfonts,amsthm}
\usepackage{mathrsfs}
\usepackage{enumerate}
\usepackage{xcolor}

\newcommand{\MCchange}[2]{{\color{blue}#1}{\color{red}#2}}

\usepackage{hyperref}

\makeatletter

\renewcommand\subsection{%
  \@ifstar{\subsection@star}{\subsection@nostar}%
}

\newcommand{\subsection@nostar}[1]{%
  \par\addvspace{.70\baselineskip}%
  \refstepcounter{subsection}%
  \addcontentsline{toc}{subsection}{\protect\numberline{\thesubsection}#1}%
  \noindent{\normalfont\normalsize \thesubsection.\quad{\bfseries #1.}}%
  \par\nobreak\vspace{.05\baselineskip}%
  \@afterindentfalse\@afterheading
}

\newcommand{\subsection@star}[1]{%
  \par\addvspace{.70\baselineskip}%
  \noindent{\normalfont\normalsize{\bfseries #1.}}%
  \par\nobreak\vspace{.05\baselineskip}%
  \@afterindentfalse\@afterheading
}

\renewcommand\subsubsection{%
  \@ifstar{\subsubsection@star}{\subsubsection@nostar}%
}

\newcommand{\subsubsection@nostar}[1]{%
  \par\addvspace{.55\baselineskip}%
  \refstepcounter{subsubsection}%
  \addcontentsline{toc}{subsubsection}{\protect\numberline{\thesubsubsection}#1}%
  \noindent{\normalfont\normalsize \thesubsubsection.\quad{\bfseries #1.}}%
  \par\nobreak\vspace{.05\baselineskip}%
  \@afterindentfalse\@afterheading
}

\newcommand{\subsubsection@star}[1]{%
  \par\addvspace{.55\baselineskip}%
  \noindent{\normalfont\normalsize{\bfseries #1.}}%
  \par\nobreak\vspace{.05\baselineskip}%
  \@afterindentfalse\@afterheading
}

\makeatother

\numberwithin{equation}{section}
\newtheorem{theorem}{Theorem}[section]
\newtheorem{lemma}[theorem]{Lemma}
\newtheorem{proposition}[theorem]{Proposition}
\newtheorem{corollary}[theorem]{Corollary}
\newtheorem{hypothesis}[theorem]{Hypothesis}
\newtheorem{characterisation}[theorem]{Characterisation}

\theoremstyle{definition}

\newtheorem{definition}[theorem]{Definition}
\newtheorem{example}[theorem]{Example}

\theoremstyle{remark}
\newtheorem{remark}[theorem]{Remark}

\newcommand{\R}{\mathbb R}
\newcommand{\C}{\mathbb C}
\newcommand{\N}{\mathbb N}
\newcommand{\cQ}{\mathcal Q}
\newcommand{\cG}{\mathcal G}

\newcommand{\Dom}{\operatorname{Dom}}
\newcommand{\Ker}{\operatorname{Ker}}

\newcommand{\norm}[1]{\left\lVert #1\right\rVert}

\newcommand{\LL}{\mathcal{L}}
\newcommand{\inner}[2]{\left\langle #1,#2\right\rangle}
\newcommand{\F}{\mathcal{F}}

\newcommand{\PP}{\mathcal{P}}

\title[Generic simplicity for self-adjoint operators]{%
Generic simplicity for self-adjoint operators under bounded potential perturbations}

\author[MARIANNA CHATZAKOU AND BERNARD HELFFER]{%
MARIANNA CHATZAKOU AND BERNARD HELFFER}

\thanks{The first author is a postdoctoral fellow of the Research Foundation
\MCchange{–}{--} Flanders (FWO) under the postdoctoral grant No 1210226N}

\address{Department of Mathematics: Analysis, Logic and Discrete Mathematics, Ghent University, 9000, Ghent, Belgium}
\email{marianna.chatzakou@ugent.be}

\address{Laboratoire de Math\'ematiques Jean Leray, CNRS, Nantes Universit\'e, F44000. Nantes, France}
\email{Bernard.Helffer@univ-nantes.fr}

\begin{document}

\begin{abstract}
We are interested in the generic simplicity of the spectrum of self-adjoint
operators under bounded potential perturbations. More precisely, given a
semibounded self-adjoint operator with compact resolvent and a suitable
space of real-valued bounded perturbations, we study whether all eigenvalues
of the perturbed operator are simple for a generic choice of the potential.
In the first part of this paper we prove an abstract criterion which ensures
that the set of perturbations giving only simple eigenvalues is residual. In the second part, we apply this criterion to several geometric and analytic
settings, including sub-Laplacians and maximally hypoelliptic operators on
compact manifolds, Laplacians on bounded domains with different boundary
conditions, and Schr\"odinger-type operators on non-compact spaces.
\end{abstract}



\maketitle

\section{Introduction}\label{sec:introduction}

\subsection{Motivation and abstract strategy}
The question whether the spectrum of an operator is simple after a generic zeroth-order perturbation is a classical one.  In the elliptic case on compact manifolds it is closely related to the work of Albert and Uhlenbeck on generic properties of eigenfunctions \cite{Alb71,Alb73,Alb75,Alb78,U}. In particular, the work of Uhlenbeck \cite{U} makes use of many techniques that are not available for all the models considered here.  In that setting one usually perturbs an elliptic operator by a smooth real potential and proves that, for a residual set of such potentials, all eigenvalues are simple.  Once simplicity is known, several nodal consequences can then be proved.  For example, the usual variational proof of Courant's nodal theorem gives the sharp bound for each eigenfunction when the corresponding eigenvalue is simple.  This viewpoint is also useful in the sub-Riemannian setting, where   Pleijel and Courant type results for sub-Laplacians are considered, see \cite{FHCourant, EL, Al, FHPleijel}. 
The present paper isolates the  perturbative mechanism behind generic simplicity.  We do not start from an elliptic operator, and we do not use smoothness of eigenfunctions in the proof of the abstract theorem.  The base operator $L$ is only assumed to be self-adjoint on \(L^2(X)\), semibounded from below, and with compact resolvent.  The perturbations are real bounded multiplication operators.  More precisely, they are taken from an admissible real Fr\'echet space \(\mathcal Q\), see Definition~\ref{def:admissible}, continuously embedded in \(L^\infty(X,\mathbb R)\).  For \(q\in\mathcal Q\) we write
\[
        L_q=L+M_q,
\]
and we ask whether the set of \(q\)'s for which all eigenvalues of \(L_q\) are simple is residual in \(\mathcal Q\).

The answer is affirmative under a finite-dimensional splitting condition on the eigenspaces.  More precisely, if \(\lambda\) is a multiple eigenvalue of \(L_q\) and \(E_q=\ker(L_q-\lambda I)\), we assume that there is a perturbation \(\sigma\in\mathcal Q\) such that the finite-dimensional operator $G_\sigma:=(P_{E_{q}}M_\sigma)|_{E_q}$, where $P_{E_{q}}:L^2(X) \rightarrow E_q$ is the orthogonal projection on $E_q$ with respect to the $L^2(X)$ inner product, and $M_\sigma$ is the multiplication by $\sigma$,  satisfies the hypothesis 
\begin{equation}
G_\sigma:=(P_{E_q}M_\sigma)|_{E_q}:E_q\to E_q,
\qquad
G_\sigma\notin \mathbb{R}I_{E_q}.
\tag{\normalfont H$_{q,\lambda,\sigma}$}
\end{equation}

The point of the paper is that  if the condition \(C_c^\infty(X^\circ,\mathbb R)\subset\mathcal Q\) is satisfied,  then no unique-continuation argument is needed to find such a \(\sigma\). In particular, this permits to include the case of sub-Laplacians with $C^{\infty}(M)$ potentials.  The perturbative part of the proof is close to the idea of Albert \cite{Alb75} based on Kato's analytic perturbation theory \cite{Kato}. We point out that, in contrast to \cite{Alb75}, the existence of a perturbation
\(\sigma \in C_c^\infty(X^\circ,\mathbb{R})\) such that
\(G_\sigma \notin \mathbb{R}I_{E_q}\) can be established without assuming the
smoothness of the associated eigenfunctions; it follows from a simple
contradiction argument based on the Cauchy--Schwarz inequality. Finally, we point out that Uhlenbeck's work \cite{U} also considers properties of eigenfunctions that do
not fit into the current framework.

\subsection{Relation with known results}
Our initial goal was to extend the elliptic results of \cite{Alb75}.  In particular, we cover the cases of  smooth perturbations on compact manifolds, bounded perturbations on possibly nonsmooth bounded domains, and bounded smooth perturbations in non-compact examples for which the unperturbed operator already has compact resolvent. A further point of the present approach, in contrast with the classical elliptic perturbation arguments of Albert \cite{Alb75}, is that the finite-dimensional splitting argument is carried out directly in the complex Hilbert space \(L^{2}(X,\mathbb C)\), and does not require the eigenspaces of the operators under consideration to be spanned by real-valued eigenfunctions. About the assumption on the compact resolvent, we note that this is part of the abstract theorem: bounded potentials preserve the compactness of the resolvent, but in general, they do not create it; see  for instance the case of the harmonic oscillator. We also mention the work of Luzzini and Zaccaron~\cite{LZ}, who studied
Maxwell eigenvalues in electromagnetic cavities under variations of the
electric permittivity and proved, among other sensitivity results, generic
simplicity of the positive Maxwell spectrum with respect to matrix-valued
permittivity perturbations.

The recent work of Yang and Guo \cite{YangGuo2026} studies a degenerate elliptic Dirichlet problem of the form
\[
        -\operatorname{div}(w\nabla u)=\lambda u,
\]
where the weight is positive inside the domain and degenerates on part of the boundary.  Their paper proves a Courant nodal-domain theorem in that setting and also obtains a residual simplicity statement for bounded potentials.   We note that the  Dirichlet realization of \(-\operatorname{div}(w\nabla)\), under their compact-embedding assumption on the weighted form domain, is a self-adjoint semibounded operator with compact resolvent; the perturbations \(\rho\in L^\infty(\Omega)\) are precisely bounded multiplication perturbations.  Thus their residual simplicity theorem fits into the present framework with \(X=\Omega\), \(L=-\operatorname{div}(w\nabla)\), and \(\mathcal Q=L^\infty(\Omega,\mathbb R)\).   Our proof separates the abstract spectral perturbation argument from the operator-specific questions, such as the description of the form domain, compactness, regularity, and unique continuation. Regarding the study of the  nodal domains of the aforesaid operator, the authors of    \cite{YangGuo2026} make use of the fact that nodal domains are  linked to    the variational theory of
spectral minimal partitions as studied  in the  work of Helffer,
Hoffmann-Ostenhof and Terracini \cite{HHOT}.

An important class of operators to which our abstract theorem applies is that of sub-Riemannian Laplacians $\mathcal L_{\rm sub}$.   For such operators,  associated with vector fields satisfying H\"ormander's condition, hypoellipticity and sub-elliptic estimates go back to H\"ormander \cite{Hormander67}, while optimal estimates were considered by Rothschild--Stein \cite{RothschildStein}. Related spectral control questions for degenerate sub-elliptic models were recently studied by Harakeh and Hillairet for Baouendi--Grushin type operators in  \cite{HarakehHillairet}.  In the  sub-Riemannian context, the proof of the Pleijel-type bounds of Frank and Helffer uses nilpotent approximation, Weyl asymptotics and Faber--Krahn inequalities \cite{FHCourant} in the equi-regular case of rank $2$.  Our paper uses their restriction theorem \cite[Theorem~2.2]{FHCourant} in one application to Courant's theorem for sub-Laplacians, proving in particular that the set of potentials $q \in C^\infty$ for which the operator $\mathcal L_{\rm sub}+q$ satisfies Courant's nodal theorem is residual.

Another class of examples comes from polynomials of vector fields, which are maximally hypoelliptic operators.  A characterisation of these operators was conjectured by Helffer and Nourrigat in 1979, see  \cite{HN79,HN85} and finally was proven in 2022 by Androulidakis, Mohsen and Yuncken \cite{AMY}. Although it is not necessary, we follow their presentation here which involves a calculus adapted to filtered foliation (see also the recent abstract $C^*$-algebraic approach of Mohsen~\cite{Mohsen2026}). 

Several further classes of self-adjoint Schr\"odinger and magnetic Schr\"odinger operators on \(\mathbb R^d\) with compact resolvent are furnished by the classical criteria and examples in \cite{Simon1983,HelfferMohamed1988,Iwatsuka1986,KondratievShubin2002,Shen1995,Guibourg1993,HelfferNourrigat2019}. 
Whenever these results give a self-adjoint, semibounded operator \(L\) with compact resolvent, our abstract theorem applies to \(L+M_q\) for bounded real perturbations \(q\in \mathcal{Q}\), for instance when \(C_c^\infty(\mathbb R^d,\mathbb R)\subset \mathcal{Q}\subset L^\infty(\mathbb R^d,\mathbb R)\).

In \cite{U}, K. Uhlenbeck also considers metric perturbations for Laplacians, and we hope to extend this
result to the sub-Riemannian case \cite{CH1}.

\subsection{Main theorem and applications}

In the sequel we denote by $L: \Dom(L)\subset H \rightarrow H$, where $H$ is a Hilbert space, the symmetric operator with domain $\Dom(L)$. We say that $L$ is semibounded (from below) if there exists $C \in \mathbb R$ such that, for all $u \in \Dom(L)$
\[
\inner{Lu}{u}\ge C \norm{u}^2\,.
\]
 In this case we simply write $L \geq C$.

\noindent More precisely, in the sequel we  assume that for every \(q\in\mathcal Q\) in some suitable function space $\mathcal{Q}$, the operator  \(L_q:=L+M_q\), where $M_q$ stands for the multiplication operator by $q$, is
self-adjoint, semibounded, and has compact resolvent. We denote by
\[
\lambda_1(q)\le \lambda_2(q)\le \cdots
\]
the eigenvalues of \(L_q\), counted with multiplicity and arranged in
non-decreasing order.

The main theorem is Theorem~\ref{thm:abstract}.  It states that, if \(L\) is self-adjoint, semibounded, has compact resolvent, and the pair \((L,\mathcal Q)\) satisfies Hypothesis~\textup{(H)}, i.e. for every $q\in\mathcal Q$ and every multiple eigenvalue $\lambda\in\sigma(L_q)$ there exists $\sigma\in\mathcal Q$ such that \textup{(H$_{q,\lambda,\sigma}$)} holds, then
\[
        \mathcal G_{\mathcal Q}
        =\{q\in\mathcal Q:\text{ all eigenvalues of } L+M_q\text{ are simple}\}
\]
is residual in \(\mathcal Q\).  Since \(\mathcal Q\) is a Fr\'echet space, it is completely metrizable; hence it is a Baire space by the Baire category theorem for complete metric spaces, and the residual set is dense.

We also discuss the finite-dimensional splitting criterion when perturbations are required to be supported in a fixed open set \(\omega\subset X^\circ\).  In that case, the criterion is the non-vanishing of the restrictions of the eigenspace to \(\omega\), and a sufficient condition is the weak unique-continuation property from \(\omega\).  This localisation is not needed for the global residual theorem when all compactly supported perturbations are available, but it is useful when one wants to prescribe where the perturbation is made.

The applications are collected in Section~\ref{sec:applications}.  On compact manifolds without boundary we treat sub-Laplacians and maximally hypoelliptic operators.  For sub-Laplacians, generic simplicity implies a residual form of Courant's nodal theorem for the perturbed operator.  On bounded domains we include Dirichlet and Neumann Laplacians on Lipschitz domains, Robin Laplacians on arbitrary open sets of finite measure in the sense of Arendt and Warma \cite{ArendtWarma}, and magnetic Schr\"odinger realizations. We also include the case of operators which are sums of squares of H\"ormander's vector fields with continuous uniformly
elliptic coefficient matrices, and the case of weighted divergence-form elliptic operators on bounded
domains.  On non-compact spaces we include harmonic and anharmonic oscillators and magnetic Schr\"odinger operators, where the base operator has compact resolvent.

\section{Abstract framework}\label{sec:abstract}

Throughout this section, $X$ is a smooth manifold, possibly with boundary, endowed with a fixed positive smooth density $d\mu$, and
\[
L^2(X):=L^2(X,\C; d\mu).
\]
When $X$ has boundary, $X^\circ$ denotes its interior.

\subsection{Conditions and preliminary results}
\begin{definition}[Admissible perturbation space]\label{def:admissible}
 An admissible perturbation space on \(X\) is a real Fr\'echet space
\((\mathcal Q,\tau_{\mathcal Q})\) together with a continuous injective linear map
\[
\iota:\mathcal Q\longrightarrow L^\infty(X,\mathbb R).
\]
By abuse of notation, we identify \(q\in\mathcal Q\) with its image
\(\iota(q)\in L^\infty(X,\mathbb R)\).
All topological notions on \(\mathcal Q\) are understood with
respect to the Fr\'echet topology \(\tau_{\mathcal Q}\).
\end{definition}

\begin{hypothesis}\label{def:splitting}
Let \(L:\Dom(L)\subset L^2(X)\to L^2(X)\) be self-adjoint, let
\(\mathcal Q\) be an admissible perturbation space, and for \(q\in\mathcal Q\)
write
\[
L_q:=L+M_q.
\]
For \(q\in\mathcal Q\) and for an eigenvalue $\lambda$ of $L_q$, we set
\[
E_{q,\lambda}:=\Ker(L_q-\lambda I).
\]
We say that the pair \((L,\mathcal Q)\) satisfies Hypothesis~\textup{(H)} if,
for every \(q\in\mathcal Q\) and every eigenvalue \(\lambda\) such
that \(\dim E_{q,\lambda}\ge2\), there exists \(\sigma\in\mathcal Q\) such that the operator $G_{\sigma}$ satisfies \textup{(H$_{q,\lambda,\sigma}$)}.
\end{hypothesis}

\begin{proposition}
\label{prop:automatic-splitting}
Let $E\subset L^2(X)$ with $\dim E\ge 2$. Then, there exists \(\sigma\in C_c^{\infty}(X^\circ,\R)\) such that 
\[
\forall c \in \R \qquad  G_\sigma:=(P_E M_\sigma)|_E:E\to E \neq cI_E.
\]
\end{proposition}

\begin{proof}
Assume by contradiction that for every \(\sigma\in C_c^{\infty}(X^\circ,\R)\) one has $G_\sigma=c_\sigma I_E$
for some \(c_\sigma\in \mathbb R\).
Choose orthonormal vectors $\varphi,\psi\in E$.
Then, for every \(\sigma\in C_c^{\infty}(X^\circ,\R)\), we have
\[
0
=
\langle G_\sigma\varphi,\psi\rangle_{L^2}
=
\langle M_\sigma\varphi,\psi\rangle_{L^2}
=
\int_X \sigma\,\varphi\,\overline{\psi}\,d\mu.
\]
Since by Cauchy--Schwarz $\varphi\,\overline{\psi}\in L^1(X,\mathbb C)$, by the above it follows that
\[
\varphi\,\overline{\psi}=0
\qquad\text{a.e. on }X.
\]
On the other hand, because \(G_\sigma=c_\sigma I_E\), we have
\[
\langle G_\sigma\varphi,\varphi\rangle_{L^2}
=
\langle G_\sigma\psi,\psi\rangle_{L^2}
\qquad\text{for every }\sigma\in C_c^{\infty}(X^\circ,\R),
\]
implying that 
\[
\int_X \sigma\bigl(|\varphi|^2-|\psi|^2\bigr)\,d\mu=0
\qquad\text{for every }\sigma\in C_c^{\infty}(X^\circ,\R),
\]
and so $|\varphi|^2=|\psi|^2$ a.e. on $X$. Combining the above we get $\varphi=\psi=0$ a.e. on $X$, which contradicts $\|\varphi\|_{L^2(X)}=\|\psi\|_{L^2(X)}=1$, finishing the proof. 
\end{proof}

\begin{corollary}
\label{cor:automatic-splitting}
If $C_c^\infty(X^\circ,\mathbb R)\subset \mathcal Q$, then the pair \((L,\mathcal Q)\)
satisfies hypothesis {\rm (H)}.
\end{corollary}

\begin{proof}
This follows from
Proposition~\ref{prop:automatic-splitting} applied to $E=\Ker(L_q-\lambda I)$.

\end{proof}

\begin{remark}
\label{rem:localized-ucp}
Let \(\omega\subset X^\circ\) be a non-empty open set and set
\(
Q_\omega:=C_c^\infty(\omega,\mathbb R).
\)
For a finite-dimensional subspace \(E\subset L^2(X)\), define
\(
E_\omega:=\{u|_\omega:\ u\in E\}\subset L^2(\omega).
\)
Then, if \(\dim E\geq 2\),
\[
\exists \sigma\in Q_\omega
\quad\text{such that}\quad
(P_E M_\sigma)|_E\notin \mathbb R I_E
\]
if and only if
\[
E_\omega\neq \{0\}.
\]
Indeed, if \(E_\omega=\{0\}\), then clearly
\(
(P_E M_\sigma)|_E=0\in \mathbb R I_E.
\)
Conversely, assume that \(E_\omega\neq\{0\}\), and suppose that
\[
(P_E M_\sigma)|_E\in \mathbb R I_E
\qquad\text{for every }\sigma\in C_c^\infty(\omega,\mathbb R).
\]
Choose \(u\in E\) with \(u|_\omega\neq0\) and \(\|u\|_{L^2(X)}=1\), and choose
\(v\in E\cap u^\perp\) with \(\|v\|_{L^2(X)}=1\). Repeating the proof of
Proposition~\ref{prop:automatic-splitting}, with test functions supported in \(\omega\), gives
\[
u\overline v=0
\qquad\text{and}\qquad
|u|^2=|v|^2
\quad\text{a.e. on }\omega.
\]
Hence \(u=v=0\) a.e. on \(\omega\), contradicting \(u|_\omega\neq0\).
Therefore there exists \(\sigma\in C_c^\infty(\omega,\mathbb R)\) such that
\[
(P_E M_\sigma)|_E\notin \mathbb R I_E .
\]
In particular, for an eigenspace
\(
E_q=\ker(L_q-\lambda I),
\)
set
\(
E_{q,\omega}:=\{u|_\omega:u\in E_q\}\subset L^2(\omega).
\)
If, in addition, \(C_c^\infty(\omega,\mathbb R)\subset \mathcal{Q}\), then the preceding
criterion shows that the perturbation \(\sigma\) in Hypothesis~{\rm (H)} can be
chosen with support in \(\omega\) provided
\[
E_{q,\omega}\neq\{0\}.
\]
A simple sufficient condition for this is the weak unique-continuation
property from \(\omega\):
\[
u\in\ker(L_q-\lambda I),
\qquad
u=0\text{ a.e. on }\omega
\quad\Longrightarrow\quad
u=0\text{ a.e. on }X.
\tag{\(\mathrm{UCP}_\omega\)}
\]
Indeed, if \(R_\omega:E_q\to L^2(\omega)\) is the restriction map
\(
R_\omega u=u|_\omega,
\)
then \((\mathrm{UCP}_\omega)\) says precisely that
\[
\ker R_\omega=\{0\}.
\]
Thus \(R_\omega\) is injective. Hence, if \(E_q\neq\{0\}\), then
\(
E_{q,\omega}=R_\omega(E_q)\neq\{0\}.
\)
Consequently, when \(\dim E_q\ge2\), there exists
\(
\sigma\in C_c^\infty(\omega,\mathbb R)
\)
such that
\[
(P_{E_q}M_\sigma)|_{E_q}\notin \mathbb R I_{E_q}.
\]
\end{remark}

\begin{lemma}
\label{lem:bounded-perturbations}
Let $H$ be a complex Hilbert space, and let
\[
L:\Dom(L)\subset H\to H
\]
be self-adjoint, semibounded, and with compact resolvent. Let $B\in \mathcal B(H)$ be bounded and self-adjoint. Then $L+B$ is self-adjoint on $\Dom(L)$, semibounded, and has compact resolvent.
\end{lemma}

\begin{proof} Obviously $L+B$ has domain $\Dom(L)$ and is symmetric. By \cite[Chapter V, \S 4, Theorem 4.3]{Kato}
 $L+B$ is also  self-adjoint on $\Dom(L)$. Now since  $L\ge -CI$, for some $C \in \R$ and $B \in \mathcal B(H)$, we also get $L+B\ge -(C+\norm{B})I$; i.e., $L+B$ is also semibounded. Now, to prove that $L+B$ has compact resolvent, choose $z<-(C+\|B\|)$.  Then $z\in \rho(L+B)$, and also  $z\in \rho(L)$. We may therefore apply the resolvent identity:
\[
(L+B-z)^{-1}-(L-z)^{-1}
=
-(L+B-z)^{-1}B(L-z)^{-1}.
\]
Since $L$ has compact resolvent, the operator $(L-z)^{-1}$ is compact. Moreover the operators $B, (L+B-z)^{-1}$ are bounded, and so the right-hand side of the above equality is a compact operator. Hence, the operator $(L+B-z)^{-1}$ is also compact, and this finishes the proof. 
\end{proof}

\begin{lemma}
\label{lem:lipschitz}
Let $L: \Dom(L) \subset L^2(X) \rightarrow L^2(X)$ be self-adjoint, semibounded, and with compact resolvent. Let $\cQ$ be an admissible perturbation space. Then for every $p,q\in \cQ$ and every $k\ge 1$,
\[
|\lambda_k(p)-\lambda_k(q)|\le \norm{p-q}_{L^\infty(X)}.
\]
\end{lemma}

\begin{proof}
To simplify the notation we will denote by  $Q_A$ the closed quadratic form of an operator $A$.  Since $p,q \in \cQ$ are bounded, we have   $\Dom(Q_L)=\Dom(Q_{L_p})= \Dom(Q_{L_q})$. For $u \in \Dom(Q_L)$ with $\norm{u}_{L^2}=1$ one has
\[
Q_{L_p}(u)=Q_{L}(u)+\int_X p|u|^2\,d\mu,
\qquad
Q_{L_q}(u)=Q_{L}(u)+\int_X q|u|^2\,d\mu.
\]
Hence
\[
|Q_{L_p}(u)-Q_{L_q}(u)|
=\left|\int_X (p-q)|u|^2\,d\mu\right|
\le \norm{p-q}_{L^\infty(X)},
\]
and so 
\[
Q_{L_q}(u)-\norm{p-q}_{L^\infty}
\le Q_{L_p}(u)
\le Q_{L_q}(u)+\norm{p-q}_{L^\infty}.
\]
The assumptions on $L$ allow us to  apply the min--max principle (see e.g \cite[Theorem~XIII.1]{RS4}) and get 
\[
\lambda_k(q)-\norm{p-q}_{L^\infty}
\le \lambda_k(p)
\le \lambda_k(q)+\norm{p-q}_{L^\infty},
\]
which proves the claim.
\end{proof}

\begin{proposition}
\label{lem:kato}
 Let
\(
L:\Dom(L)\subset L^2(X)\to L^2(X)
\)
be self-adjoint, semibounded and  with compact resolvent. Let \(p,\sigma\in \cQ\), and for $t \in \R$, define $T(t):=L+M_p+tM_\sigma$.
Let \(\lambda\) be an eigenvalue of \(T(0)=L+M_p\) of multiplicity
\[
m:=\dim E,
\qquad
E:=\ker(T(0)-\lambda I).
\]
Then there exist \(t_0>0\), real-analytic functions $\{\mu_j\}_{j\ge 1},\mu_j:(-t_0,t_0)\to \mathbb R $ and $\{\nu_j\}_{j\ge 1}, \nu_j:(-t_0,t_0)\to L^2(X,d\mu) $
such that:

\begin{enumerate}
\item for every fixed \(|t|<t_0\), one has $T(t)\nu_j(t)=\mu_j(t)\nu_j(t)$ for all $j\ge 1$
and the family \(\{\nu_j(t)\}_{j\ge 1}\) forms a complete orthonormal basis of
\(L^2(X)\);

\item after (possible) relabelling, one has $\mu_j(0)=\lambda$ for $j=1,\dots,m,$
and the set \(\{\nu_1(0),\dots,\nu_m(0)\}\) forms an orthonormal basis of \(E\);

\item for every \(j=1,\dots,m\), one has the Feynman--Hellmann formula $G_\sigma \nu_j(0)=\mu_j'(0)\,\nu_j(0).$
In particular, the numbers $\mu_1'(0),\dots,\mu_m'(0)$
are precisely the eigenvalues of \(G_\sigma\), counted with multiplicity.
\end{enumerate}
\end{proposition}

\begin{proof}
Consider the complexified family $\widetilde T(z):=L+M_p+zM_\sigma, z\in\mathbb C.$ 
Since \(p,\sigma\in \mathcal{Q}\subset L^\infty(X,\mathbb R)\), the multiplication operators
\(M_p\) and \(M_\sigma\) are bounded on \(L^2(X)\). Hence the family
\(z\mapsto \widetilde T(z)\) has constant domain \(\Dom(L)\), and for every
\(u\in \Dom(L)\) the map
\[
z\longmapsto \widetilde T(z)u=(L+M_p)u+zM_\sigma u
\]
is entire. Therefore \(\widetilde T\) is a holomorphic family of type \textup{(A)} in the sense of \cite[Chapter VII, \S 2]{Kato}.
Moreover, $\widetilde T(z)^*=\widetilde T(\overline z)$, and in particular, for every \(t\in \mathbb R\), the operator $T(t)=\widetilde T(t)$
is self-adjoint on \(\Dom(L)\), since \(M_p+tM_\sigma\) is bounded and self-adjoint. We next show that \(\widetilde T(z)\) has compact resolvent for every \(z\in\mathbb C\).
Set $B_z:=M_p+zM_\sigma$. Then $B_z \in \mathcal{B}(L^2(X))$.  Since \(L\) is self-adjoint, \(i\tau\in\rho(L)\) for every \(\tau\in\mathbb R\setminus\{0\}\), and $\|(L-i\tau)^{-1}\|\le |\tau|^{-1}.$
Hence choosing  \(\tau\in\mathbb R\setminus\{0\}\) such that \(|\tau|>\|B_z\|\) we obtain 
\[
\|B_z(L-i\tau)^{-1}\|<1,
\]
and therefore $\widetilde T(z)-i\tau$
is invertible, with
\[
(\widetilde T(z)-i\tau)^{-1}
=
(L-i\tau)^{-1}\bigl(I+B_z(L-i\tau)^{-1}\bigr)^{-1}.
\]
Since \((L-i\tau)^{-1}\) is compact, it follows that
\((\widetilde T(z)-i\tau)^{-1}\) is compact. After choosing \(t_0>0\) sufficiently small, and set \(I_0:=(-t_0,t_0)\). An application of
\cite[Chapter~VII, Theorem~3.9]{Kato} yields scalar-valued functions $\widetilde\lambda_1,\widetilde\lambda_2,\dots,$
and \(L^2(X)\)-valued functions $\widetilde\phi_1,\widetilde\phi_2,\dots,$
all holomorphic on \(I_0\), such that for every
\(t\in I_0\),
\[
T(t)\widetilde\phi_n(t)=\widetilde\lambda_n(t)\widetilde\phi_n(t),
\qquad
\|\widetilde\phi_n(t)\|_{L^2(X)}=1,
\qquad
\langle \widetilde\phi_n(t),\widetilde\phi_k(t)\rangle_{L^2(X)}=\delta_{nk},
\]
and the family \(\{\widetilde\phi_n(t)\}_{n\ge1}\) is a complete orthonormal basis of
\(L^2(X)\). Moreover, the list \(\{\widetilde\lambda_n(t)\}_{n\ge1}\) exhausts the spectrum of
\(T(t)\), counted with multiplicity.
Since \(\lambda\) has multiplicity \(m\) for \(T(0)\), after relabelling the family if
necessary, we may assume that
\[
\widetilde\lambda_j(0)=\lambda
\qquad\text{for }j=1,\dots,m,
\]
and
\[
\widetilde\lambda_j(0)\neq \lambda
\qquad\text{for }j\ge m+1.
\]
For every \(j\ge 1\), define $\mu_j(t):=\widetilde\lambda_j(t)$ and $\nu_j(t):=\widetilde\phi_j(t)$ for $t\in I_0$.
Then the functions \(\mu_j,\nu_j\) are real-analytic on \(I_0\) and satisfy
\[
T(t)\nu_j(t)=\mu_j(t)\nu_j(t),
\qquad
\|\nu_j(t)\|_{L^2(X)}=1,
\qquad
\langle \nu_j(t),\nu_k(t)\rangle_{L^2(X)}=\delta_{jk},
\]
and the family \(\{\nu_j(t)\}_{j\ge1}\) is a complete orthonormal basis of \(L^2(X)\) for every
\(t\in I_0\). Moreover,
\[
\mu_j(0)=\lambda
\qquad\text{for }j=1,\dots,m.
\]
At \(t=0\), the vectors \(\nu_1(0),\dots,\nu_m(0)\) are orthonormal eigenvectors of
\(T(0)\) corresponding to the eigenvalue \(\lambda\). Since \(\lambda\) has
multiplicity \(m\), they form an orthonormal basis of \(E\).
It remains to prove \textup{(3)}.
Fix \(j\in\{1,\dots,m\}\), and let \(e\in E\) be arbitrary. Taking the inner product of
\[
T(t)\nu_j(t)=\mu_j(t)\nu_j(t)
\]
with \(e\), we get
\[
\langle T(t)\nu_j(t),e\rangle_{L^2(X)}
=
\mu_j(t)\,\langle \nu_j(t),e\rangle_{L^2(X)}.
\]
Since $T(t)=T(0)+tM_\sigma,$
this becomes
\[
\langle T(0)\nu_j(t),e\rangle_{L^2(X)}
+
t\,\langle M_\sigma\nu_j(t),e\rangle_{L^2(X)}
=
\mu_j(t)\,\langle \nu_j(t),e\rangle_{L^2(X)}.
\]
Now \(e\in E\) means $T(0)e=\lambda e,$
and since \(T(0)\) is self-adjoint, we have
\[
\langle T(0)\nu_j(t),e\rangle_{L^2(X)}
=
\langle \nu_j(t),T(0)e\rangle_{L^2(X)}
=
\lambda\,\langle \nu_j(t),e\rangle_{L^2(X)}.
\]
Hence
\[
t\,\langle M_\sigma\nu_j(t),e\rangle_{L^2(X)}
=
(\mu_j(t)-\lambda)\,\langle \nu_j(t),e\rangle_{L^2(X)}.
\]
For \(t\neq 0\), divide by \(t\) and let \(t\to 0\). Since \(\mu_j\)
is real-analytic,
\[
\frac{\mu_j(t)-\lambda}{t}\longrightarrow \mu_j'(0).
\]
Since \(\nu_j\) is \(L^2(X)\)-valued real-analytic, we have $\nu_j(t)\longrightarrow \nu_j(0)$ and $M_\sigma\nu_j(t)\longrightarrow M_\sigma\nu_j(0)$ in $L^2(X)$.
Therefore
\[
\langle M_\sigma\nu_j(0),e\rangle_{L^2(X)}
=
\mu_j'(0)\,\langle \nu_j(0),e\rangle_{L^2(X)}
\qquad \text{for every }e\in E.
\]
Since \(\nu_j(0)\in E\), this is equivalent to
\[
\langle P_E M_\sigma\nu_j(0),e\rangle_{L^2(X)}
=
\mu_j'(0)\,\langle \nu_j(0),e\rangle_{L^2(X)}
\qquad \text{for every }e\in E.
\]
Hence
\[
P_E M_\sigma\nu_j(0)=\mu_j'(0)\,\nu_j(0),
\]
that is,
\[
G_\sigma \nu_j(0)=\mu_j'(0)\,\nu_j(0).
\]
Thus each \(\nu_j(0)\), \(j=1,\dots,m\), is an eigenvector of \(G_\sigma\) with
eigenvalue \(\mu_j'(0)\). Since $\nu_1(0),\dots,\nu_m(0)$
is an orthonormal basis of \(E\), the numbers $\mu_1'(0),\dots,\mu_m'(0)$
are exactly the eigenvalues of \(G_\sigma\), counted with multiplicity.
\end{proof}

\begin{remark}[Comparison with the transversality hypothesis of Arnold]
    Let
\(
        E_q=\ker(L+M_q-\lambda I).
\)
For \(a\in \mathcal{Q}\), set
\(
        H_a:=L+M_{q+a}.
\)
Thus \(H_0=L+M_q\).  Y. Colin de Verdi\`ere as in \cite{CDV88}, considers
the map
\[
        \Phi:a\longmapsto \Phi(a),
\]
where \(\Phi(a)\) is the quadratic form induced by \(H_a\) on the eigenspace
corresponding to the eigenvalue issuing from \(\lambda\).  Following the terminology of \cite{CDV88} we call the following hypothesis on $\Phi$ the transversality hypothesis of Arnold, which asks whether the differential
\(
        d\Phi_0
\)
quadratic forms on \(E_q\); see
\cite[pp.~185--186]{CDV88}. In the present potential perturbation setting, the differential of \(\Phi\)
at \(0\), in the direction \(\sigma\in \mathcal{Q}\), is
\[
        d\Phi_0(\sigma)[u,v]
        =
        \frac{d}{dt}\bigg|_{t=0}
        \left\langle (L+M_q+tM_\sigma)u,v\right\rangle_{L^2}
        =
        \left\langle M_\sigma u,v\right\rangle_{L^2},
        \qquad u,v\in E_q,
\]
or, equivalently, 
\[
        d\Phi_0(\sigma)[u,v]
        =
        \left\langle G_\sigma u,v\right\rangle_{L^2},
        \qquad u,v\in E_q .
\]
Therefore, if the strong Arnold hypothesis holds, then \(d\Phi_0\) is
surjective onto all quadratic forms on \(E_q\) and it is easy to be seen that the latter implies that for all $\sigma \in \mathcal{Q}$ we have  $ G_\sigma\notin \mathbb R I_{E_q}$.
Thus the strong Arnold hypothesis implies our hypothesis
\textup{(H$_{q,\lambda,\sigma}$)}.  Our hypothesis is weaker because it only asks for
one such non-scalar direction \(\sigma\), and not for the surjectivity of
\(d\Phi_0\) onto all quadratic forms on \(E_q\).
\end{remark}

\begin{lemma}
    \label{lem:An-open}
    Let
\(
L:\Dom(L)\subset L^2(X)\to L^2(X)
\)
be self-adjoint, semibounded and  with compact resolvent, and let $q \in \cQ$. For $n \in \N$, we also define the sets 
\[
A_n:=\{q\in\mathcal Q:\lambda_1(q),\dots,\lambda_n(q)\text{ are simple}\},
\qquad A_0:=\mathcal Q.
\]
 Then for every $n \geq 1$, the set $A_n$ is open in $\cQ$. 
\end{lemma}

\begin{proof}
    By Lemma~\ref{lem:bounded-perturbations}, every $L_q$ is self-adjoint on $\Dom(L)$, bounded from below, and has compact resolvent. Hence its spectrum consists of a non-decreasing sequence of real eigenvalues
\[
\lambda_1(q)\le \lambda_2(q)\le \cdots \nearrow +\infty,
\]
counted with multiplicity.
We claim that for each $n\ge 1$ the set  $A_n$ is open in $\tau_\cQ$. Indeed, let $p\in A_n$. Then, by definition of $A_n$, we have 
\[
\lambda_1(p)<\lambda_2(p)<\cdots<\lambda_n(p)<\lambda_{n+1}(p).
\]
Set $\delta:=2^{-1}\min_{1\le j\le n}\bigl(\lambda_{j+1}(p)-\lambda_j(p)\bigr)>0.$
If $q\in \cQ$ satisfies $\norm{p-q}_{L^{\infty}(X)}<\delta,$
then Lemma~\ref{lem:lipschitz} yields
\[
|\lambda_j(p)-\lambda_j(q)|<\delta
\qquad\text{for }1\le j\le n+1.
\]
Hence, for every $j=1,\dots,n$, we have 
\begin{align*}
\lambda_{j+1}(q)-\lambda_j(q)
&=\bigl(\lambda_{j+1}(p)-\lambda_j(p)\bigr)
+\bigl(\lambda_{j+1}(q)-\lambda_{j+1}(p)\bigr)
-\bigl(\lambda_j(q)-\lambda_j(p)\bigr) \\
&>\bigl(\lambda_{j+1}(p)-\lambda_j(p)\bigr)-2\delta\ge0.
\end{align*}
Since the first inequality is strict, we obtain
\[
\lambda_1(q)<\lambda_2(q)<\cdots<\lambda_n(q)<\lambda_{n+1}(q).
\]
Thus $q\in A_n$, and because the  inclusion $\cQ\hookrightarrow L^\infty(X,\R)$ is continuous, we have shown that  $A_n$ is open in $\tau_\cQ$.
\end{proof}

\begin{lemma}\label{lem:reduce_mult_abstract}
 Let
\(
L:\Dom(L)\subset L^2(X)\to L^2(X)
\)
be self-adjoint, semibounded and  with compact resolvent.  Assume that the pair $(L,\cQ)$ satisfies the hypothesis {\rm (H)}. Fix \(n\ge 0\), and let
\(p\in A_n\).   Suppose that the eigenvalue
\(
\lambda:=\lambda_{n+1}(p)
\)
of \(L_p\) has multiplicity $m:=\dim \Ker(L_p-\lambda I)>1.$
Then for every neighborhood \(U\) of \(p\) in \(\cQ\), there exists
\(
q\in U\cap A_n
\)
such that the multiplicity of the eigenvalue \(\lambda_{n+1}(q)\) of \(L_q\)
is strictly smaller than the multiplicity of the eigenvalue \(\lambda_{n+1}(p)\) of \(L_p\).
\end{lemma}

\begin{proof}
Since \(\lambda=\lambda_{n+1}(p)\) has multiplicity \(m\), the ordered eigenvalues of \(L_p\)
satisfy
\[
\lambda_{n+1}(p)=\lambda_{n+2}(p)=\cdots=\lambda_{n+m}(p)=\lambda.
\]
Moreover, since \(p\in A_n\), the first \(n\) eigenvalues are simple, and hence for all $ j \in  \{1, \cdots, n \}$ we have $\lambda_j(p)<\lambda$.
Also, by definition of the multiplicity \(m\), we know $\lambda<\lambda_{n+m+1}(p).$
Choose \(\rho>0\) such that
\[
0<\rho<
\frac13
\min\Bigl(
\{\lambda-\lambda_j(p):1\le j\le n\}
\cup
\{\lambda_{n+m+1}(p)-\lambda\}
\Bigr),
\qquad n\ge1,
\]
and if \(n=0\), choose
\[
0<\rho<\frac13\bigl(\lambda_{m+1}(p)-\lambda\bigr).
\]
If we set $I:=(\lambda-\rho,\lambda+\rho)$, then we see that  \(I\) contains no point of \(\sigma(L_p)\) other than \(\lambda\).
Let
\[
E:=\Ker(L_p-\lambda I).
\]
 Since $(L,\cQ)$ satisfies  (H), there exists $\sigma\in\cQ$
such that $G_\sigma$ is not a scalar multiple of the identity.
Now let
\[
g_1\le g_2\le \cdots \le g_m
\]
be the eigenvalues of \(G_\sigma\), counted with multiplicity. Since \(G_\sigma\notin\mathbb R I_E\), not all \(g_i\) are equal; hence there exists an integer \(1\le r<m
\) such that
\[
g_1=\cdots=g_r<g_{r+1}\le \cdots \le g_m.
\]
Consider the perturbed family $T(t):=L_p+tM_\sigma.$
By Proposition~\ref{lem:kato}, after shrinking to a sufficiently small interval $(-t_0,t_0)$
there exist real-analytic functions
\[
\mu_1(t),\dots,\mu_m(t)
\]
such that
\[
\mu_j(0)=\lambda,
\qquad
\mu_j'(0)=g_j,
\qquad j=1,\dots,m,
\]
and these are precisely the eigenvalues of
\(T(t)\) lying in the interval \(I\), counted with multiplicity, i.e., there are at least $m$ eigenvalues of $T(t)$, $t \neq 0$, inside $I$. Now, for every pair \(1\le j\le r<k\le m\), we have
\[
\mu_k(t)-\mu_j(t)
=
(\mu_k'(0)-\mu_j'(0))t+o(t)
=
(g_k-g_1)t+o(t).
\]
Setting 
\(
c:=\min_{k>r}(g_k-g_1)>0,
\)
we see that there exists \(t_1>0\) such that for every
\(
0<t<t_1
\)
and every such pair \((j,k)\), one has
\(
\mu_k(t)-\mu_j(t)>\frac{c}{2}\,t>0.
\)
Hence, for fixed \(0<t<t_1\),
\[
\mu_j(t)<\mu_k(t)
\qquad\text{whenever }1\le j\le r<k\le m,
\]
and therefore, for fixed $0<t<t_1$,  the smallest value among $\mu_1(t),\dots,\mu_m(t)$
can only be attained among $\mu_1(t),\dots,\mu_r(t).$
Equivalently, if
\begin{equation}
    \label{eq:nu(t)}
    \nu(t):=\min\{\mu_1(t),\dots,\mu_m(t)\}
\end{equation}
and
\begin{equation}
    \label{eq:J(t)}
    J(t):=\{\,j\in\{1,\dots,m\}: \mu_j(t)=\nu(t)\,\},
\end{equation}
then
\[
J(t)\subset \{1,\dots,r\},
\qquad\text{hence}\qquad
\#J(t)\le r.
\]
Since \(U\) is a neighborhood of \(p\) in \(\cQ\), and since the map
\(
t\longmapsto p+t\sigma
\)
is continuous from \(\mathbb R\) to \(\cQ\), there exists \(t_2>0\) such that
\[
p+t\sigma\in U
\qquad\text{whenever }|t|<t_2.
\]
Since, by Lemma \ref{lem:An-open}, the set  \(A_n\) is open, there exists \(t_3>0\) such that
\[
p+t\sigma\in A_n
\qquad\text{whenever }|t|<t_3.
\]
If \(n\ge 1\),  set
\[
\delta_1:=(\lambda-\rho)-\lambda_n(p)>0,
\qquad
\delta_2:=\lambda_{n+m+1}(p)-(\lambda+\rho)>0.
\]
By Lemma \ref{lem:lipschitz}, we have 
\[
|\lambda_k(p+t\sigma)-\lambda_k(p)|
\le
\|t\sigma\|_{L^\infty(X)}
=
|t|\,\|\sigma\|_{L^\infty(X)}
\qquad\text{for all }k\ge 1.
\]
Since \(\sigma\in\cQ\subset L^\infty(X)\), and $\sigma \neq 0$,  we may choose
\[
t_4:=\frac{1}{2\|\sigma\|_{L^\infty(X)}}\min\{\delta_1,\delta_2\}>0.
\]
Then for every \(|t|<t_4\) and for $n \geq 1$ we have
\[
\lambda_n(p+t\sigma)<\lambda-\rho,
\qquad
\lambda_{n+m+1}(p+t\sigma)>\lambda+\rho.
\]
Indeed, if $n\ge 1$, then for $t_4$ as above we have 
\(
|t|\,\|\sigma\|_{L^\infty(X)}<\delta_1,\delta_2
\)
implying that 
\[
\lambda_n(p+t\sigma)
\le
\lambda_n(p)+|\lambda_n(p+t\sigma)-\lambda_n(p)|
<
\lambda_n(p)+\delta_1
=
\lambda-\rho,
\]
and also that 
\[
\lambda_{n+m+1}(p+t\sigma)
\ge
\lambda_{n+m+1}(p)-|\lambda_{n+m+1}(p+t\sigma)-\lambda_{n+m+1}(p)|
>
\lambda_{n+m+1}(p)-\delta_2
=
\lambda+\rho.
\]
If \(n=0\), only the second inequality is needed, and one may take
\(
t_4:=\frac{\delta_2}{2\|\sigma\|_{L^\infty(X)}}.
\)
Now fix
\[
0<t<t_{\min}:=\min\{t_0,t_1,t_2,t_3,t_4\},
\]
and set
\(
q:=p+t\sigma.
\)
Since $t<t_2,t_3$, we have 
\(
q\in U\cap A_n\,,
\)
and since $t<t_4$, 
\(
\lambda_n(q)<\lambda-\rho
\)
if $n \ge 1$, and
\(
\lambda_{n+m+1}(q)>\lambda+\rho.
\)
Moreover, for  $t<t_0$, the interval
\(
I:=(\lambda-\rho,\lambda+\rho)
\)
contains at least  the $m$ eigenvalues
\[
\mu_1(t),\dots,\mu_m(t)
\]
of the operator
\[
T(t)=L_p+tM_\sigma=L_q,
\]
counted with multiplicity. Actually, we will show that $I$ contains exactly $m$ eigenvalues of $L_q$: it is easy to see that there are exactly $n$ eigenvalues of $L_q$ below the interval $I$.
Indeed, if $n\ge 1$, then $\lambda_n(q)<\lambda-\rho$, so there are at least $n$
eigenvalues below $I$; on the other hand, if there were more than $n$ eigenvalues below $I$,
then together with the (at least) $m$ eigenvalues lying in $I$ this would give at least $n+m+1$
eigenvalues strictly below $\lambda+\rho$, contradicting
\[
\lambda_{n+m+1}(q)>\lambda+\rho.
\]
Hence there are exactly $n$ eigenvalues of $L_q$ below $I$, and since also $\lambda_{n+m+1}(q) \notin I$, we see that there are exactly $m$ eigenvalues of $L_q$ inside $I$. The case $n=0$ is simpler.
Therefore the eigenvalues of $L_q$ lying in $I$ are precisely
\[
\lambda_{n+1}(q),\lambda_{n+2}(q),\dots,\lambda_{n+m}(q),
\]
counted with multiplicity, and so, for a fixed $0<t<t_{\min}$, one has  $\{\lambda_{n+1}(q), \cdots,\lambda_{n+m}(q) \}=\{\mu_1(t),\cdots, \mu_{m}(t)\}$ as multisets. In particular,
\(
\lambda_{n+1}(q)
\)
is the smallest eigenvalue of $L_q$ lying in $I$.
Now, for a fixed $0<t<t_{\min}$, consider $\nu(t)$ as in \eqref{eq:nu(t)}.
Since
\[
\mu_j(t)<\mu_k(t)
\qquad\text{for all }1\le j\le r<k\le m,
\]
the minimum $\nu(t)$ can only be attained among
\(
\mu_1(t),\dots,\mu_r(t).
\)
Since the multiset of eigenvalues of $L_q$ lying in $I$ is exactly
\[
\{\mu_1(t),\dots,\mu_m(t)\},
\]
the multiplicity of the eigenvalue $\nu(t)$ of $L+q$ is exactly $\#J(t)$ where $J(t)$ is as in \eqref{eq:J(t)}.
But $\nu(t)$ is the smallest eigenvalue of $L_q$ in $I$, so
\[
\nu(t)=\lambda_{n+1}(q).
\]
Therefore
\[
\operatorname{mult}\bigl(\lambda_{n+1}(q)\bigr)=\#J(t)\le r<m.
\]
Hence the multiplicity of \(\lambda_{n+1}(q)\) is strictly smaller than that of
\(\lambda_{n+1}(p)\), and the proof is complete.
\end{proof}

\begin{lemma}\label{lem:dense}
 Let
\(
L:\Dom(L)\subset L^2(X)\to L^2(X)
\)
be self-adjoint, semibounded and  with compact resolvent.  Assume that the pair $(L,\cQ)$ satisfies the hypothesis {\rm (H)}. Then, for every $n\ge 0$, the set $A_{n+1}$ is dense in $A_n$.
\end{lemma}

\begin{proof}
Fix $p\in A_n$ and let $U$ be a neighborhood of $p \in \cQ$. If $\lambda_{n+1}(p)$ is already simple,
then $p\in U\cap A_{n+1}$ and there is nothing to prove. Otherwise, apply Lemma~\ref{lem:reduce_mult_abstract} to find
$q_1\in U\cap A_n$ such that the multiplicity of $\lambda_{n+1}(q_1)$ is strictly smaller than that of
$\lambda_{n+1}(p)$. If $\lambda_{n+1}(q_1)$ is still not simple, then $q_1\in A_n$ and $A_n$ is open, so
$U\cap A_n$ is an open neighborhood of $q_1$; applying Lemma~\ref{lem:reduce_mult_abstract} again, we get
$q_2\in U\cap A_n$ with strictly smaller multiplicity for $\lambda_{n+1}(q_2)$. This process terminates after finitely many steps, producing some
\[
q\in U\cap A_n
\]
for which $\lambda_{n+1}(q)$ is simple. Then $q\in U\cap A_{n+1}$, which proves that $A_{n+1}$ is dense in $A_n$.
\end{proof}

\begin{remark}
\label{rem:localized-density}
In view of Remark \ref{rem:localized-ucp}, assume that  \(\omega\subset X^\circ\) is open and 
\[
C_c^\infty(\omega,\mathbb R)\subset \mathcal{Q}.
\]
If \((\mathrm{UCP}_\omega)\) holds for every eigenspace of every operator
\(L_q\), then the density step in Lemma~\ref{lem:dense} can be performed with
perturbations supported in \(\omega\).  More precisely, for every \(n\geq0\),
every \(p\in A_n\), and every neighborhood \(\mathcal U\) of \(p\) in \(\mathcal{Q}\),
there exists
\(
q\in \mathcal U\cap A_{n+1}
\)
such that
\(
q-p\in C_c^\infty(\omega,\mathbb R).
\)
Equivalently,
\(
\operatorname{supp}(q-p)\Subset\omega.
\)
\end{remark}

\subsection{The main abstract theorem}

\begin{theorem}
\label{thm:abstract}
Let
\(
L:\Dom(L)\subset L^2(X)\to L^2(X)
\)
be self-adjoint, semibounded and  with compact resolvent.  Assume that the pair $(L,\cQ)$ satisfies the hypothesis {\rm (H)}.   Then the set
\[
\cG_{\cQ}:=
\Bigl\{q\in \cQ:\text{all eigenvalues of }L+M_q\text{ are simple}\Bigr\}
\]
is residual in $\cQ$.
\end{theorem}

\begin{proof}
By Lemma~\ref{lem:An-open}, each \(A_n\) is open in \(\mathcal Q\).
By Lemma~\ref{lem:dense}, \(A_{n+1}\) is dense in \(A_n\) for every
\(n\geq 0\). Since \(A_0=\mathcal Q\), it follows by induction that
each \(A_n\) is open and dense in \(\mathcal Q\). Moreover,
\[
\mathcal G_{\mathcal Q}=\bigcap_{n\geq 1}A_n .
\]
Hence \(\mathcal G_{\mathcal Q}\) is a countable intersection of open
dense subsets of \(\mathcal Q\), and is therefore residual in
\(\mathcal Q\). Since \(\mathcal Q\) is a Fr\'echet space, it is a Baire
space. Consequently, by the Baire category theorem,
\(\mathcal G_{\mathcal Q}\) is dense in \(\mathcal Q\).
\end{proof}

\begin{remark}
\label{rem:standard-settings-old}
Below are geometric settings, and corresponding admissible perturbation spaces $\cQ$ for which hypothesis {\rm (H)} is automatic:
    \begin{itemize}
        \item In the case where $X=M$ is a compact manifold $M$ without boundary, the natural choice for the admissible perturbation space is $\cQ=C^{\infty}(M,\R)$, endowed with its usual Fr\'echet topology. Clearly,
        \[
        C_c^\infty(M,\R)=C^\infty(M,\R)=\cQ.
        \]

        \item In the case where $X=M$ is a compact manifold with boundary and interior $M^\circ$, the natural choice for the admissible perturbation space is $\cQ=C^\infty(M,\R)$, endowed with its usual Fr\'echet topology. We have
        \[
        C_c^\infty(M^\circ,\R)\hookrightarrow C^\infty(M,\R)=\cQ.
        \]
        If $M$ is realised as the closure $M=\overline{\Omega}$ of a bounded regular domain $\Omega\subset \R^n$, one may equivalently write
        \[
        \cQ=C^\infty(\overline{\Omega},\R),
        \]
        and then
        \[
        C_c^\infty(\Omega,\R)\hookrightarrow C^\infty(\overline{\Omega},\R).
        \]

        \item In the case where $X=\Omega\subset\R^n$ is a bounded domain, possibly with non-regular boundary, a basic choice for the admissible perturbation space is
        \[
        \cQ=L^\infty(\Omega,\R),
        \]
        endowed with the norm
        \[
        \norm{q}_{L^\infty(\Omega)}:=\operatorname*{ess\,sup}_{x\in\Omega}|q(x)|.
        \]
        We have
        \[
        C_c^\infty(\Omega,\R)\hookrightarrow L^\infty(\Omega,\R)=\cQ.
        \]
        Thus in the non-smooth boundary setting one does not impose any regularity up to the boundary on the perturbation.

        \item In the case of a non-compact domain $X \subset \R^n$ the basic example for $\cQ$ is the space
        \[
        C_b^\infty(X,\R)
        :=
        \{q\in C^\infty(X,\R):\partial^\alpha q\in L^\infty(X)\text{ for every multi-index }\alpha\},
        \]
        endowed with the seminorms
        \[
        \norm{q}_m:=\max_{|\alpha|\le m}\norm{\partial^\alpha q}_{L^\infty(X)}.
        \]
        We have
        \[
        C_c^\infty(X,\R)\hookrightarrow C_b^\infty(X,\R)=\cQ.
        \]
    \end{itemize}

    Therefore, in all the above settings, Proposition \ref{prop:automatic-splitting} shows that hypothesis {\rm (H)} is automatically satisfied. In particular, in the regular boundary case one may work with perturbations that are smooth up to the boundary, whereas in the non-regular boundary examples of Section~\ref{sec:Mwithboubndary} one works with bounded perturbations in $L^\infty(\Omega,\R)$.  In particular, if the operator
    \[
    L:\Dom(L)\subset L^2(X)\to L^2(X)
    \]
    is self-adjoint, semibounded and with compact resolvent, then the abstract Theorem \ref{thm:abstract} applies to the pair \((L, \mathcal Q)\).
\end{remark}

\section{Applications}\label{sec:applications}

In this section we apply the abstract theorem, Theorem~\ref{thm:abstract}, to several classes of operators on compact manifolds without boundary, compact manifolds with boundary, bounded domains, and non-compact manifolds.

\subsection{On compact manifolds without boundary}\label{subsec:non-elliptic}

Here the underlying setting is a compact manifold $M$ without boundary and the space of admissible perturbations is $\cQ=C^{\infty}(M,\R)$.
\subsubsection{Sub-Laplacians and Courant's nodal theorem}

For a fixed positive smooth density $d\mu$  the sub-Laplacian $\LL_{\rm sub}$ on $M$ is defined by
\[
\mathcal L_{\rm sub}:=\sum_{i=1}^m X_i^*X_i\,,
\]
where  $X_1,\dots,X_m\in \mathfrak X^\infty(M)$ are smooth vector fields on $M$, and $X^*$ stands for the formal adjoint of  $X\in \mathfrak X^\infty(M)$ defined by
\[
\int_M (Xu)vd\mu=\int_M u(X^{*}v)d\mu,
\qquad \text{for all }u,v\in C^{\infty}(M,\R).
\]
One then sees that $\LL_{\rm sub}$  satisfies
\[
(\mathcal L_{\rm sub}u,u)_{L^2(M)}=\sum_{i=1}^m \norm{X_i u}_{L^2(M)}^2,
\qquad u\in C^{\infty}(M,\R)\,.
\]
The vector fields $X_1,\dots,X_m\in \mathfrak X^\infty(M)$ satisfy the so-called H\"ormander condition introduced in \cite{Hormander67}. Let $k \in \N^{*}$ denote the smallest number of iterated commutators of the vector fields $X_1,\dots,X_m$ that is needed to generate $T_xM$ at any $x \in M$.  H\"ormander proved that  the operator $\mathcal L_{\rm sub}$ is hypoelliptic  \cite[Theorem~1.1]{Hormander67} and also provided a sub-elliptic estimate with loss. Later on, Rothschild and Stein \cite[Theorem~17 and estimate~(17.20), p.~311]{RothschildStein} proved the following optimal  sub-elliptic estimate: there exists $C>0$ such that for every $u\in C^{\infty}(M,\R)$,
\begin{equation}\label{eq:sub-elliptic-estimate}
\norm{u}_{H^{2/k}(M)}^2 \le C\left(\norm{\mathcal L_{\rm sub}u}_{L^2(M)}^2 + \norm{u}_{L^2(M)}^2\right)\,.
\end{equation}
To be precise, the sub-elliptic estimate \eqref{eq:sub-elliptic-estimate} appears in \cite{RothschildStein} in a local form. However, it is not difficult to globalise the argument on the whole $M$, see e.g. \cite[Theorem 1.4]{LaurentLeautaud}. 

Moreover, since $\mathcal L_{\rm sub}$ is symmetric and nonnegative,
hypoellipticity of $\mathcal L_{\rm sub}+I$ and compactness of $M$ imply that $\mathcal L_{\rm sub}$ is essentially self-adjoint;
see \cite[Theorem~X.26]{RS2}. Hence it admits a self-adjoint extension with 
\[
\mathcal L_{\rm sub}: \Dom(\mathcal{L}_{\rm sub}) \subset L^2(M) \longrightarrow L^2(M)\,.
\]
We keep the notation $\mathcal L_{\rm sub}$ for its self-adjoint extension, and we note that $\LL_{\rm sub}$ has compact resolvent, and there exists an orthonormal basis of eigenfunctions
$(\varphi_j)_{j\in\N}$ of $L^2(M)$ and a nondecreasing sequence of real eigenvalues $(\lambda_j)_{j\in\N}$, multiplicity counted,  such that
\[
\mathcal L_{\rm sub}\varphi_j=\lambda_j\varphi_j,
\qquad
(\varphi_i,\varphi_j)_{L^2(M)}=\delta_{ij},
\qquad
0=\lambda_0<\lambda_1\le \lambda_2\le \cdots \to +\infty\,.
\]
 Moreover, each $\varphi_j \in C^{\infty}(M,\R)$ is smooth. Indeed, since $\varphi_j \in L^2(M)$ and so also $\LL_{\rm sub} \varphi_j \in L^2(M)$, a repeated  use of  the sub-elliptic estimate \eqref{eq:sub-elliptic-estimate} shows that $\varphi_j \in H^s(M)$ for arbitrary large $s$, and so $\varphi_j \in C^{\infty}(M,\R)$. Similar arguments show that if \(q\in C^\infty(M,\mathbb R)\), and  \(u\in L^2(M)\)  satisfies $(\LL_{\rm sub}+q)u=\lambda u$ for 
  some \(\lambda\in\mathbb R\), then $u\in C^\infty(M)$, i.e. the eigenfunctions of $\LL_{\rm sub}+q$ are smooth.

\begin{proposition}\label{cor:courant_sub}
Let \(\LL=\LL_{\rm sub}\) be the sub-Laplacian associated with the vector fields
\(X_1,\dots,X_m\). Then the set
\[
\mathcal C_{\rm sub}
:=
\Bigl\{q\in C^\infty(M,\mathbb R):
\text{ every eigenfunction of }\LL_{\rm sub}+M_q
\text{ satisfies Courant's nodal theorem}\Bigr\}
\]
is residual in \(C^\infty(M,\mathbb R)\).
\end{proposition}

\begin{proof} The proof follows from Theorem \ref{thm:abstract} and  \cite[Theorem 1]{EL}  which says that for any $n \in \N$, any eigenfunction of $\mathcal L_{\rm sub}$ with eigenvalue
$\lambda_k$ has at most $k+{\rm mult}(\lambda_k)-1$ nodal domains, where ${\rm mult}(\lambda_k)$ denotes the multiplicity of $\lambda_k$. Let us recall that the key point is to verify in the sub-Riemannian case the so-called restriction lemma, see  \cite[Theorem 2.2]{FHCourant} and references therein.
\end{proof}

\subsubsection{Maximally hypoelliptic differential operators}

 We fix vector fields
\(
X_1,\ldots,X_m\in \mathfrak X^\infty(M)
\)
satisfying H\"ormander's bracket-generating condition.  We denote by
\(\mathcal F^\bullet\) the bracket filtration generated by these vector
fields in the sense of \cite[Definition 1.3]{AMY}.  Thus
\(
\mathcal F^0:=0,
\)
and, for \(j\geq 1\), we have 
\[
\mathcal F^j
:=
\operatorname{span}_{C^\infty(M)}
\left\{
[X_{i_1},[X_{i_2},\ldots,[X_{i_{\ell-1}},X_{i_\ell}]\ldots]]
:
1\leq \ell\leq j
\right\},
\]
with the convention that brackets of length \(1\) are just the vector fields
\(X_1,\ldots,X_m\).  In particular,
\[
\mathcal F^1
=
\operatorname{span}_{C^\infty(M)}\{X_1,\ldots,X_m\}.
\]
There exists \(N\in\mathbb N\) such that
\(
\mathcal F^N=\mathfrak X^\infty(M).
\)
Moreover,
\[
[\mathcal F^i,\mathcal F^j]\subset \mathcal F^{i+j},
\qquad i,j\geq 0,
\]
so \(\mathcal F^\bullet=(\mathcal F^1,\ldots,\mathcal F^N)\) is a filtered
foliation.

Although the operator is written as a polynomial in the first-layer vector
fields \(X_j\in\mathcal F^1\), the associated nilpotent model is determined
by the full bracket filtration \(\mathcal F^\bullet\).

The intrinsic definition of the Helffer--Nourrigat cone used here follows the modern formulation in \cite{AMY}; its original form can be found in \cite[p.~13]{HN79}. 

\begin{definition}[Helffer--Nourrigat cone]  For \(x\in M\), we  set
\[
I_x:=\{f\in C^\infty(M):f(x)=0\},
\qquad
\F_x^j:=\F^j/I_x\F^j,
\]
and define
\[
\mathfrak g_x
:=
\bigoplus_{j=1}^N \F_x^j/\F_x^{j-1}.
\]
Then $\mathfrak g_x$ is a graded nilpotent Lie algebra, and we denote by $G_x$ the associated simply connected
graded nilpotent Lie group. Following \cite[\S~1.4]{AMY}, if  $X \in \mathcal{X}^{\infty}(M)$ we denote by $\widetilde X \in  \mathcal{X}^{\infty}(M \times \R_{+})$ the vector field $\widetilde X(x,t)=(X(x),0)$.  The  adiabatic foliation associated to $\mathcal{F}^{\bullet}$  is
\[
a\F^\bullet
=
t\widetilde{\F}^1+t^2\widetilde{\F}^2+\cdots+t^N\widetilde{\F}^N,
\]
and the Helffer--Nourrigat cone at \(x\) is
\[
T_x^\ast\F^\bullet
:=
\left\{
\xi\in \mathfrak g_x^\ast :
(x,\xi,0)\in
\overline{T^\ast M\times \mathbb R_+^\ast}^{\,(a\F^\bullet)^\ast}
\right\}.
\]
Via Kirillov's orbit method, more precisely through the canonical identification
\(
G_x\backslash \mathfrak g_x^\ast \simeq \widehat{G_x},
\)
the image
\(
T_x^\ast\mathcal F^\bullet/G_x
\subset G_x\backslash \mathfrak g_x^\ast
\)
may be regarded as the corresponding closed dilation-invariant subset of the unitary dual
\(\widehat{G_x}\).
\end{definition}

\begin{characterisation}[Maximal hypoellipticity in the compact case, {\cite[Theorem C]{AMY}}]\label{char:maximally}
     Let $D: C^{\infty}(M, \C) \rightarrow C^{\infty}(M,\C)$ be a non-commutative polynomial of the form $ D
 =
 \sum_{|I|\leq k} a_I X_I,$ where $a_I\in C^\infty(M,\mathbb C)$. Then $D$ is called a maximally hypoelliptic operator if one of the following equivalent conditions is satisfied:
     \begin{itemize}
         \item for every (or for some) $s \in \R, D:  H_{X}^{s+k}(M) \rightarrow  H_{X}^{s}(M)$ is invertible modulo compact operators;
         \item for every (or for some) $s \in \R$ and for every distribution $ u: Du \in  H_{X}^s(M) \Rightarrow u \in  H_{X}^{s+k}(M)$;
         \item for every $x\in M$ and every non-trivial $\pi\in T_x^*\mathcal{F}^\bullet$, the principal symbol $\sigma_k(D,x,\pi)$ is injective on smooth vectors $C^\infty(\pi)$.
     \end{itemize}

    \end{characterisation}
   
Recall that the Sobolev space $H_{X}^{s}$, for $s \in \N$, is defined as 
\[
H_{X}^{s}(M):= \bigl\{\, u \in L^2(M):  X_I u \in L^2(M) \,\,\text{for all}\,\, |I|\leq s \, \bigr\}.
\]
The definition extends to any $s \in \R$ by interpolation ($s>0$) or  duality ($s<0$). 

\begin{proposition}\label{prop:maxhyp-apply}
  Let $D$ satisfy Characterisation~\ref{char:maximally}. Assume moreover that $D$ is formally self-adjoint on $C^\infty(M,\mathbb C)$ and  bounded from below.  Then, the  set
\[
\mathcal G_{D}
:=
\Bigl\{q\in C^\infty(M,\mathbb R):
\text{ all eigenvalues of }D+q\text{ are simple}\Bigr\}
\]
is residual in $C^\infty(M,\mathbb R)$.
\end{proposition}

\begin{proof}
 Since $D$ is maximally hypoelliptic,  \cite[Theorem 3.43(a)]{AMY} gives   injectivity of its principal symbol. Since $D$ is formally symmetric, the same applies to the formal adjoint, and  \cite[Theorem 3.44(b)]{AMY} implies that $D$ is  $*$-maximally hypoelliptic. Hence \cite[Corollary~3.45(a)]{AMY} implies that $D$ is essentially self-adjoint
on $L^2(M)$. We keep the same notation for its self-adjoint extension. Since $D$ is bounded from below, choose   $z<\inf \sigma(D)$.  Then $z\in \rho(D)$ is an element of the resolvent set of $D$, and so the operator
\[
D-z:\Dom(D)\subset L^2(M)\to L^2(M)
\]
is bijective. We now show that
\begin{equation}\label{eq:inv.D-z}
(D-z)^{-1}:L^2(M)\to H_X^k(M)
\end{equation}
is bounded. Let $f\in L^2(M)= H_{X}^0(M)$, and set $u:=(D-z)^{-1}f.$ Then $(D-z)u=f\in  H_{X}^0(M).$
Since  the operator $D-z$ trivially has the same principal symbol
as $D$ it is also maximally hypoelliptic by
\cite[Theorem~3.43(a)]{AMY}, and so by  Characterisation \ref{char:maximally} $u\in  H_{X}^k(M).$
Thus $(D-z)^{-1}$ is well-defined as a map
\[
L^2(M)\to  H_{X}^k(M).
\]
To prove boundedness, let $f_n\to f$ in $L^2(M)$, and assume that
\[
(D-z)^{-1}f_n\to u
\qquad\text{in } H_{X}^k(M).
\]
Set $u_n:=(D-z)^{-1}f_n.$ Since the inclusion $H_{X}^k(M)\hookrightarrow L^2(M)$ is continuous, we also have $u_n\to u$ in $L^2(M)$.
Moreover,
\[
(D-z)u_n=f_n\to f
\qquad\text{in }L^2(M).
\]
Because $D-z$ is closed (being self-adjoint), it follows that
\[
u\in \Dom(D)
\qquad\text{and}\qquad
(D-z)u=f.
\]
Thus $u=(D-z)^{-1}f$, and the graph of \eqref{eq:inv.D-z} is closed. By the closed graph theorem 
$(D-z)^{-1}$ is bounded.
Now by \cite{Hormander67} the inclusion \(
 H_{X}^k(M)\hookrightarrow L^2(M)
\) is compact. A combination of the latter together with the fact that the operator \eqref{eq:inv.D-z} is bounded, implies that $D$ has compact resolvent. The conclusion then follows by Theorem \ref{thm:abstract}, and the proof is complete. 
\end{proof}

\subsubsection{Sums of squares with continuous coefficients}
The next example shows that, for H\"ormander-type divergence-form operators with uniformly elliptic coefficients, the abstract generic-simplicity theorem applies and one also obtains Courant's nodal domain bound for the generic potentials.

\begin{proposition}\label{prop:courant_horizontal_elliptic}
Let \(X_1,\dots,X_m\in \mathfrak X^\infty(M)\) be a H\"ormander system of
smooth vector fields on a compact manifold \(M\) without boundary, and let
\[
        A=(a_{ij})_{1\leq i,j\leq m}
        \in C^0(M;\operatorname{Sym}_m(\mathbb R)).
\]
Assume that \(A\) is uniformly elliptic in the \(X\)-directions in the sense of
\cite[p.~352]{Tralli2012}; namely, there exist constants \(0<c\leq C\) such that
\[
        c|\xi|^2
        \leq
        \sum_{i,j=1}^m a_{ij}(x)\xi_i\xi_j
        \leq
        C|\xi|^2,
        \qquad x\in M,\quad \xi\in\mathbb R^m .
\]
On \(H_X^1(M)\) we define the  form
\[
        \mathfrak q_{L_A}(u,v)
        :=
        \sum_{i,j=1}^m
        \int_M a_{ij}(x)(X_j u)\overline{X_i v}\,d\mu .
\]
We denote by
\(
        L_A:\operatorname{Dom}(L_A)\subset L^2(M)\longrightarrow L^2(M)
\)
the self-adjoint operator associated with $ \mathfrak q_{L_A}$. Formally,
\[
        L_A=\sum_{i,j=1}^m X_i^*a_{ij}(x)X_j .
\]
Then the set
\[
        \mathcal C_A
        :=
        \Bigl\{
        q\in C^\infty(M,\mathbb R):
        \text{  real-valued eigenfunctions of }L_A+M_q
        \text{ satisfy Courant's theorem}
        \Bigr\}
\]
is residual in \(C^\infty(M,\mathbb R)\).
\end{proposition}

\begin{proof}
By the H\"ormander condition, the embedding
\[
        H_X^1(M)\hookrightarrow L^2(M)
\]
is compact. The uniform ellipticity of \(A\) gives
\[
        c\sum_{j=1}^m\|X_j u\|_{L^2(M)}^2
        \leq
        \mathfrak q_{L_A}(u,u)
        \leq
        C\sum_{j=1}^m\|X_j u\|_{L^2(M)}^2,
\]
so that the form norm is equivalent to the norm of \(H_X^1(M)\).  
By Proposition~\ref{prop:compact-resolvent-bounded-perturbations} and the above, $L_A$ has compact resolvent, and so by
Theorem~\ref{thm:abstract} we get that
\[
        \mathcal G_A
        :=
        \bigl\{
        q\in C^\infty(M,\mathbb R):
        \hbox{ all eigenvalues of }L_A+M_q\hbox{ are simple}
        \bigr\}
\]
is residual in \(C^\infty(M,\mathbb R)\). Next we prove that
\[
        \mathcal G_A\subset \mathcal C_A .
\]
Let \(q\in\mathcal G_A\), and let \(u\) be a real-valued eigenfunction of
\(L_A+M_q\) associated with the \(k\)-th eigenvalue \(\lambda_k(q)\). Thus
\[
        u\in \operatorname{Dom}(L_A)\subset H_X^1(M)
\]
and
\[
        (L_A+M_q)u=\lambda_k(q)u
        \qquad\text{in }L^2(M).
\]
Then, by the definition of \(L_A\) through the closed form,
\[
        \sum_{i,j=1}^m X_i^*\bigl(a_{ij}X_j u\bigr)
        -
        (\lambda_k(q)-q)u
        =
        0\,,
\]
in the weak sense on \(M\).
By the Harnack inequality and the  Harnack-to-continuity implication
of Chanillo--Wheeden \cite[Theorem 5.3]{CW86}, see also Lu's
H\"ormander-vector-field framework \cite{Lu92} and his observation
\cite[Remark 4]{Lu1994},  weak solutions are continuous; hence \(u\in C^0(M)\). See also  \cite[Theorem 3.3]{FFZ11} for the extension of the latter results. Thus, for real-valued eigenfunctions, the nodal domains are well-defined as the
connected components of
\[
        M\setminus u^{-1}(\{0\}).
\] Finally, using the restriction theorem \cite[Theorem 2.2]{FHCourant}, and reasoning as before for the case of $\mathcal{L}_{\rm sub}$, we complete the proof.
\end{proof}

\subsection{On compact manifolds with boundary and  bounded domains}\label{sec:Mwithboubndary}
In this section  we include  examples of operators on non-smooth domains where Theorem \ref{thm:abstract} applies verbatim.  We start with an example of a class of elliptic operators of second order where the set of bounded potentials for which the corresponding eigenfunctions satisfy Courant's nodal theorem is residual.  

\subsubsection{Second order elliptic operators with H\"older coefficients}

We  consider the low-regularity elliptic setting studied by
Alessandrini~\cite{Al}. The point of this formulation is that the 
coefficients are only H\"older continuous. In dimension \(n\geq 3\), this is
below the regularity needed for the 
  proof of Courant's theorem based on unique continuation.

\begin{proposition}\label{cor:courant_alessandrini}
Let \(\Omega\subset\mathbb R^n\), \(n\geq 2\), be a bounded domain.
Let \(0<\alpha<1\), and assume that
\[
        A=(a_{ij})_{1\leq i,j\leq n}
        \in C^{0,\alpha}(\Omega;\operatorname{Sym}_n(\mathbb R))
\]
is real symmetric and uniformly elliptic. Thus there exist constants
\(0<\kappa\leq K\) such that
\[
        \kappa |\xi|^2
        \leq
        A(x)\xi\cdot \xi
        \leq
        K|\xi|^2,
        \qquad
        \xi\in\mathbb R^n,\quad \text{for a.e. }x\in\Omega .
\]
Let \(b,\rho\in L^\infty(\Omega,\mathbb R)\), with $\rho\geq \rho_0>0$ a.e. in $\Omega$.
Set
\[
        L^2_\rho(\Omega, \mathbb R):=L^2(\Omega, \mathbb R;\rho dx).
\]
Let
\[
        L_{\operatorname{ell}}
        :
        \operatorname{Dom}(L_{\operatorname{ell}})
        \subset L^2_\rho(\Omega,\mathbb R)
        \longrightarrow
        L^2_\rho(\Omega, \mathbb R)
\]
be the self-adjoint operator associated with the real form
\[
        \mathfrak a_0(u,v)
        :=
        \int_\Omega A\nabla u\cdot \nabla v\,dx
        +
        \int_\Omega b\,uv\,dx,
        \qquad
        \operatorname{Dom}(\mathfrak a_0)=H^1_0(\Omega,\mathbb R).
\]
Equivalently, in the weak sense,
\[
        L_{\operatorname{ell}}u
        =
        \rho^{-1}
        \bigl(-\operatorname{div}(A\nabla u)+bu\bigr).
\]
Then the set
\[
\mathcal C_{\operatorname{ell}}
:=
\Bigl\{
V\in L^\infty(\Omega,\mathbb R):
\text{  eigenfunctions in  } H^1_0(\Omega, \mathbb R) \text{  of   }
L_{\operatorname{ell}}+M_V
\text{ satisfy Courant's theorem}
\Bigr\}
\]
is residual in \(L^\infty(\Omega,\mathbb R)\).
\end{proposition}
\begin{proof}
For \(V\in L^\infty(\Omega,\mathbb R)\), the operator
\(L_{\operatorname{ell}}+M_V\) is associated in \(L^2_\rho(\Omega, \mathbb R)\) with the
form
\[
        \mathfrak a_V(u,v)
        :=
        \mathfrak a_0(u,v)
        +
        \int_\Omega Vuv\,\rho dx,
        \qquad
        \operatorname{Dom}(\mathfrak a_V)=H^1_0(\Omega,\mathbb R).
\]
Equivalently,
\[
        \mathfrak a_V(u,v)
        =
        \int_\Omega A\nabla u\cdot \nabla v\,dx
        +
        \int_\Omega (b+\rho V)uv\,dx .
\]
The assumptions on \(A\), \(b\), and \(\rho\) imply that
\(\mathfrak a_0\) is closed and semibounded from below. Indeed, after adding
a sufficiently large multiple of the \(L^2_\rho\)-norm, the form norm is
equivalent to the usual \(H^1_0(\Omega, \mathbb R)\)-norm. Since \(\Omega\) is bounded,
the embedding
\[
        H^1_0(\Omega, \mathbb R)\hookrightarrow L^2_\rho(\Omega, \mathbb R)
\]
is compact. Hence \(L_{\operatorname{ell}}\) has compact resolvent. Applying
Theorem~\ref{thm:abstract} to the complexification of
\(L_{\operatorname{ell}}\) in \(L^2_\rho(\Omega,\mathbb C; \rho dx)\), and using
that the real and complex eigenspaces have the same multiplicities, gives
that the set
\[
        \mathcal G_{\operatorname{ell}}
        :=
        \Bigl\{
        V\in L^\infty(\Omega,\mathbb R):
        \text{ all eigenvalues of }L_{\operatorname{ell}}+M_V
        \text{ are simple}
        \Bigr\}
\]
is residual in \(L^\infty(\Omega,\mathbb R)\).
Note that, although the abstract theorem was stated using a fixed smooth density, it can also be applied directly in the weighted space
\(L^2_\rho(\Omega, \mathbb C; \rho dx)\). Next, we prove that
\[
        \mathcal G_{\operatorname{ell}}
        \subset
        \mathcal C_{\operatorname{ell}}.
\]
Let \(V\in\mathcal G_{\operatorname{ell}}\), and let \(u \in H_{0}^{1}(\Omega, \mathbb R)\) be an
eigenfunction of \(L_{\operatorname{ell}}+M_V\). Assume that \(u\) is
associated with the \(k\)-th eigenvalue, i.e., we have 
\[
        (L_{\operatorname{ell}}+M_V)u=\lambda_k(V)u,
\]
where the eigenvalues are counted with multiplicity. 
Then \(u\in H^1_0(\Omega, \mathbb R)\), and
\[
        \mathfrak a_V(u,\varphi)
        =
         \lambda_k(V)\int_\Omega u\varphi\,\rho dx,
        \qquad
        \varphi\in H^1_0(\Omega,\mathbb R).
       \]
       Equivalently,
\[
        -\operatorname{div}(A\nabla u)
        +
         d_{\lambda,V}u
        =
        0
        \qquad\text{in }\mathcal D'(\Omega),
\]
where $d_{\lambda,V}
        :=
        b+\rho V-\lambda_k(V)\rho \in L^\infty(\Omega,\mathbb R).$ Thus \(u\) solves a divergence-form uniformly elliptic equation with bounded
zeroth-order coefficient. By the De Giorgi--Nash--Moser interior H\"older
regularity theorem, see for instance Gilbarg--Trudinger
\cite[Theorem~8.24]{GT}, there exists
\(\beta\in(0,1)\) such that
\[
        u\in C^{0,\beta}_{\operatorname{loc}}(\Omega),
\]
and in particular, the nodal domains of \(u\) are well-defined as
open connected components of \(\Omega\setminus u^{-1}(\{0\})\).
By the nodal-domain restriction argument of Pleijel, in the low-regularity
form used by Alessandrini, see \cite[Theorem~1]{Plei56} and
\cite[Proposition~2.1 and Remark~2.2]{Al}, we see that  \(u\) has at most
\(k\) nodal domains, and this completes the proof.
\end{proof}

\subsubsection{Boundary realisations of the  Laplacian and of the magnetic Schr\"odinger operators}

\medskip

\begin{example}[Dirichlet Laplacian on bounded Lipschitz domains]\label{ex:Dir-Lapl}
   Let \(n\ge 2\), and let \(\Omega\subset\mathbb R^n\) be a bounded Lipschitz domain. For $V \in L^{\infty}(\Omega, \R)$, we denote by  \(H_{D,\Omega}=-\Delta+V\) the perturbed Dirichlet 
Laplacian with  domain 
\[
\operatorname{dom}(H_{D,\Omega})=
\{v\in H_0^1(\Omega):\Delta v\in L^2(\Omega;d^n x)\}.
\]
By  \cite[Theorem 3.8]{AGMT} the operator $H_{D,\Omega}$ is  self-adjoint, bounded from below, and has compact resolvent. Therefore the set
\[
\mathcal G_{D,\Omega}
:=
\{V\in L^\infty(\Omega, \R):\text{all eigenvalues of }H_{D,\Omega}\text{ are simple}\}
\]
is  residual in $L^{\infty}(\Omega, \R)$. 
\end{example}

\begin{example}[Neumann Laplacian on bounded Lipschitz domains]\label{ex:Neu-Lapl}
   Let \(n\ge 2\), and let \(\Omega\subset\mathbb R^n\) be a bounded Lipschitz  domain. For $V \in L^{\infty}(\Omega, \R)$, we denote by  \(H_{N,\Omega}=-\Delta+V\) the perturbed Neumann 
Laplacian with  domain 
\[
\operatorname{dom}(H_{N,\Omega})
=
\left\{
u\in H^1(\Omega)\,\middle|\,
\Delta u\in L^2(\Omega;d^n x),\;
\widetilde{\gamma}_N u = 0 \text{ in } H^{-1/2}(\partial\Omega)
\right\},
\]
where \(\widetilde{\gamma}_N\) denotes the weak Neumann trace operator
\[
\widetilde{\gamma}_N:
\left\{
u\in H^1(\Omega)\,\middle|\,\Delta u\in L^2(\Omega;d^n x)
\right\}
\longrightarrow H^{-1/2}(\partial\Omega).
\]
By  \cite[Theorem 3.9]{AGMT} the operator $H_{N,\Omega}$ is  self-adjoint, bounded from below, and has compact resolvent. Therefore the set
\[
\mathcal G_{N,\Omega}
:=
\{V\in L^\infty(\Omega, \R):\text{all eigenvalues of }H_{N,\Omega}\text{ are simple}\}
\]
is  residual in $L^{\infty}(\Omega, \R)$. 
\end{example}

\begin{example}[Robin Laplacian on an arbitrary open set of finite measure]
\label{ex:robin-arbitrary}
Let \(n\ge 2\), and let \(\Omega\subset \mathbb R^n\) be an arbitrary open set with
finite Lebesgue measure. Let \(\Gamma:=\partial\Omega\), and let \(\sigma\) denote the
\((n-1)\)-dimensional Hausdorff measure on \(\Gamma\). Following \cite[p.~17]{ArendtWarma}, define
\[
\Gamma_\sigma
:=
\{z\in \Gamma:\exists r>0 \text{ such that } \sigma(\Gamma\cap B(z,r))<\infty\},
\]
and let \(S\subset \Gamma\) be the maximal \(\sigma\)-admissible set, that is,
\(S\) is \(\sigma\)-admissible and
\[
\operatorname{Cap}_{\overline\Omega}(\Gamma_\sigma\setminus S)=0,
\]
where $\operatorname{Cap}_{\Omega}$ stands for the relative capacity with respect to $\Omega$.
Let $\beta\in L^\infty(S,\sigma)$, $\beta\ge \delta>0$. The Robin form on \(L^2(\Omega)\) is defined  by
\[
a_\beta(u,v)
=
\int_\Omega \nabla u\cdot \nabla v\,dx
+
\int_S u v\,\beta\,d\sigma,
\]
with domain
\[
D(a_\beta)
=
\Bigl\{
u\in H^1(\Omega)\cap C_c(\overline\Omega):
\int_S |u|^2\beta\,d\sigma<\infty
\Bigr\},
\]
see \cite[p.~17--18]{ArendtWarma}. Since \(\beta\,d\sigma\) is admissible,
the form \(a_\beta\) is closable, and  the operator
associated with \(a_\beta\) is selfadjoint; see
\cite[p.~18]{ArendtWarma}. Since
\[
a_\beta(u,u)
=
\int_\Omega |\nabla u|^2\,dx+\int_S |u|^2\beta\,d\sigma\ge 0,
\]
the operator
\[
L_{R,\beta}:=-\Delta_\beta
\]
is selfadjoint and bounded from below by \(0\). Moreover, by \cite[Theorem~5.1]{ArendtWarma}, the semigroup \(e^{t\Delta_\beta}\) is
ultracontractive under the assumption \(\beta\ge \delta>0\), and by
\cite[Corollary~5.2]{ArendtWarma}, if \(\Omega\) has finite measure then $e^{t\Delta_\beta}$ is compact for every $t>0$. Hence, see e.g.  \cite[Chapter~II, Theorem~4.29]{EngelNagel}, the generator \(\Delta_\beta\)
has compact resolvent, and therefore so does \(L_{R,\beta}=-\Delta_\beta\). Summarising \(L_{R,\beta}\) satisfies the operator hypotheses of Theorem \ref{thm:abstract} and so the set
\[
\mathcal G_{R,\beta,\Omega}
:=
\Bigl\{
q \in L^\infty(\Omega,\mathbb R):
\text{all eigenvalues of }L_{R,\beta}+q\text{ are simple}
\Bigr\}
\]
is residual in \(L^\infty(\Omega,\mathbb R)\).
\end{example}

\begin{remark}
If one wants the Robin condition on the whole boundary \(\partial\Omega\), rather than on
the maximal admissible subset \(S\), a sufficient assumption is that \(\Omega\) be
bounded with Lipschitz boundary. Indeed, \cite[Proposition~4.1]{ArendtWarma} states that
in this case the Hausdorff measure \(\sigma\) on \(\partial\Omega\) is admissible.
\end{remark}
\noindent In the examples below, concerning the Neumann and Dirichlet realizations of the magnetic Schr\"odinger operator in an open bounded regular domain $\Omega \subset \R^n$ we have collected the necessary properties of these operators from \cite{HelfferVienna} for the application of Theorem \ref{thm:abstract}.

\begin{example}[Neumann realisation of the magnetic Schr\"odinger operator]\label{ex:magnet.Sch.Neum}
    Let \(\Omega\subset \mathbb R^n\) be a bounded regular domain, let
\(
A=(A_1,\dots,A_n)\) be a $C^\infty$ vector field on $\overline{\Omega}$ and $V\in C^\infty(\overline\Omega,\mathbb R)$
and fix \(h>0\). Let
\[
P_{h,A,V,\Omega}
=
\sum_{j=1}^n (hD_{x_j}-A_j)^2+V(x)
\]
be the magnetic Schr\"odinger operator,
and denote by
\(
P^{N}_{h,A,V,\Omega},
\)
its Neumann realization. On the boundary $\partial \Omega$ we impose the  magnetic-Neumann boundary condition 
\[
\vec n\cdot(h\nabla-iA)u=0 \quad\text{on }\partial\Omega,
\]
\(\vec n\) denotes the normal vector to
\(\partial\Omega\). The Neumann realization $P^{N}_{h,A,V,\Omega}$ is the Friedrichs extension attached to the
quadratic form
\[
Q^N_{h,A,V,\Omega}(u)
:=
\int_\Omega \Bigl(|\nabla_{h,A}u|^2+V(x)|u(x)|^2\Bigr)\,dx,
\qquad
u\in C^\infty(\overline{\Omega};\mathbb C),
\]
with $\nabla_{h,A}=h\nabla-iA$. Since  \(\Omega\) is regular and bounded, the form domain of the operator $P^{N}_{h,A,V,\Omega}$ is
\(
V^N(\Omega)=H^1(\Omega),
\)
and the operator domain is
\[
D(P^N_{h,A,V,\Omega})
=
\Bigl\{
u\in H^2(\Omega):
\vec n\cdot (h\nabla-iA)u=0 \text{ on }\partial\Omega
\Bigr\},
\]
The operator  $P^{N}_{h,A,V,\Omega}$ is self-adjoint and we have 
\[
\int_\Omega \Bigl(|\nabla_{h,A}u|^2+V(x)|u(x)|^2\Bigr)\,dx
\ge -C\|u\|_{L^2(\Omega)}^2
\qquad \forall u\in C^\infty(\overline{\Omega}, \mathbb C),
\]
and so also  $P^N_{h,A,V,\Omega}$ is bounded from below. Moreover, in \cite[p. 8]{HelfferVienna} it is shown that $P^N_{h,A,V,\Omega}$ has compact resolvent. Therefore, by Theorem \ref{thm:abstract} the set \[
\mathcal G_N^{\mathrm{mag}}
:=
\Bigl\{
V\in  C^\infty(\overline\Omega,\mathbb R):
\text{all eigenvalues of }P^{N}_{h,A,V,\Omega}\text{ are simple}
\Bigr\}
\]
is residual in $C^\infty(\overline\Omega,\mathbb R)$. 
\end{example}

\begin{example}[Dirichlet  realisation of the magnetic Schr\"odinger operator]\label{ex:magnetic-Dir}
   Let $\Omega \subset \mathbb{R}^n$ be a bounded regular domain. For $P_{h,A,V,\Omega}$ as in Example \ref{ex:magnet.Sch.Neum} we define its Dirichlet realisation as the Friedrichs extension attached to the
quadratic form
\[
Q^{D}_{h,A,V,\Omega}(u):=
\int_\Omega \Bigl(|\nabla_{h,A}u|^2+V(x)|u(x)|^2\Bigr)\,dx,
\qquad
u\in C_0^\infty(\Omega;\mathbb C).
\]
 Its form domain is $V^D(\Omega)=H_0^1(\Omega)$, while the operator domain is $D(P^D_{h,A,V,\Omega})=H_0^1(\Omega)\cap H^2(\Omega)$.
The operator  $P^{D}_{h,A,V,\Omega}$ is self-adjoint and we have 
\[
\int_\Omega \Bigl(|\nabla_{h,A}u|^2+V(x)|u(x)|^2\Bigr)\,dx
\ge -C\|u\|_{L^2(\Omega)}^2
\qquad \forall u\in C_0^\infty(\Omega).
\]
So $P^D_{h,A,V,\Omega}$ is bounded from below.   Moreover, in \cite[p. 8]{HelfferVienna} it is shown that $P^D_{h,A,V,\Omega}$ has compact resolvent. Therefore, by Theorem \ref{thm:abstract} the set \[
\mathcal G_D^{\mathrm{mag}}
:=
\Bigl\{
V\in  C^\infty(\overline\Omega,\mathbb R):
\text{all eigenvalues of }P^{D}_{h,A,V,\Omega}\text{ are simple}
\Bigr\}
\]
is residual in $C^\infty(\overline\Omega,\mathbb R)$. 
\end{example}

\subsection{On non-compact manifolds}\label{subsec:noncompact}

\begin{example}[Harmonic and anharmonic oscillators]\label{ex:anharmonic} For $k \in \mathbb N^*$
    we define
\[
\PP_{2k}:=\big\{p:\R^n\rightarrow \R\, \text{ polynomial satisfying } \liminf_{|x|\rightarrow \infty}\frac{p(x)}{|x|^{2k}}>0 \big\}.
\]
Let $q\in \PP_{2\ell}$ and $p\in \PP_{2k}$ be real-valued. Then there exists $q_0,p_0>0$ such that
\[
q(\xi)+q_0\ge 1,
\qquad
p(x)+p_0\ge 1.
\]
Let $T_0$ be the positive self-adjoint realization on $L^2(\mathbb R^n)$ of
\[
T_0=q(D)+p(x)+q_0+p_0.
\]
By \cite[Corollary~5.4(b)]{CDR}, one has $T_0^{-1}\in S_r(L^2(\mathbb R^n))$ for some values of $r>0$ (depending on $k, \ell$),
and hence $T_0^{-1}$ is compact. Therefore $T_0$ has compact resolvent. Set
\[
T:=T_0-(p_0+q_0)I.
\]
 Then $T$ is self-adjoint, bounded from below, and has compact resolvent.
Additionally, by Theorem \ref{thm:abstract}, the set
\[
\mathcal G_{\mathrm{anh}}
:=
\Bigl\{V\in C_b^\infty(\mathbb R^n,\mathbb R):
\text{all eigenvalues of }T+V\text{ are simple}\Bigr\}
\]
is residual in $C_b^\infty(\mathbb R^n,\mathbb R)$.
 \end{example}

  \begin{example}[Magnetic Schr\"odinger operator on $\R^n$]\label{ex:magnetic-non-compact} 
  
Let
\[
P_{h,A,V,\mathbb R^n}
=
\sum_{j=1}^n (hD_{x_j}-A_j)^2+V(x),
\]
be as in Example \ref{ex:magnet.Sch.Neum} with $\Omega=\mathbb R^n$. It is clear that,  if \(V\ge -C\), the operator is semibounded, that  the form domain
is explicitly given by
\[
\mathcal V(\mathbb R^n)
=
\Bigl\{
u\in L^2(\mathbb R^n):
\nabla_{h,A}u\in L^2(\mathbb R^n),\ (V+C_0)^{1/2}u\in L^2(\mathbb R^n)
\Bigr\},
\]
and that $P_{h,A,V,\mathbb R^n}$ is essentially self-adjoint. We denote by $L_{\rm mag}$ its self-adjoint realisation.  If we also assume further that the embedding
\[
\mathcal V(\mathbb R^n)\hookrightarrow L^2(\mathbb R^n)
\]
is compact, then $L_{\rm mag}$ has compact resolvent.  We refer to \cite{Guibourg1993,HelfferNourrigat2019,Iwatsuka1986,KondratievShubin2002}, \cite[Theorem 1.1, Corollary 1.2]{HelfferMohamed1988} for general criteria.  
Consequently,
\[
        \mathcal G_{\rm mag,pol}
        :=
        \Bigl\{q\in C_b^\infty(\mathbb R^n,\mathbb R):
        \text{all eigenvalues of }L_{\rm mag}+M_q\text{ are simple}
        \Bigr\}
\]
is residual in \(C_b^\infty(\mathbb R^n,\mathbb R)\). In particular, if $V(x)\to +\infty$, then Persson's lemma gives $\sigma_{\mathrm{ess}}(P_{h,A,V,\mathbb R^n})=\varnothing$, and so $P_{h,A,V,\mathbb R^n}$ has compact resolvent. By Theorem \ref{thm:abstract} the set 
\[
\mathcal G_{\R^n}^{\mathrm{mag}}
:=
\Bigl\{
q\in  C_{b}^{\infty}(\R^n,\R):
\text{all eigenvalues of }P_{h,A,V,\mathbb R^n}+q\text{ are simple}
\Bigr\}
\]
is residual in $C_{b}^{\infty}(\R^n,\R)$.
 \end{example}

The sufficient condition \(V(x)\to +\infty\) in Example~\ref{ex:magnetic-non-compact}
is not necessary.  Let us recall two typical examples.
First, let \(L_{\rm S}\) be the Friedrichs realisation of
\[
        L_{\rm S}:=-\Delta_{x,y}+x^2y^2 .
\]
The potential \(x^2y^2\) does not tend to \(+\infty\) as
\(|(x,y)|\to\infty\), since it vanishes on the coordinate axes. Nevertheless,
this operator has compact resolvent; this is one of the examples first
 considered by Simon in \cite{Simon1983}.

Second, let \(L_{\rm mag}\) be the Friedrichs realisation of
\[
        L_{\rm mag}
        :=
        \left(-i\partial_x-y(x^2+y^2)\right)^2
        +
        \left(-i\partial_y+x(x^2+y^2)\right)^2 .
\]
Writing \(r^2=x^2+y^2\), and using the convention \(D+A\), this corresponds to
\(A=(-yr^2,xr^2)\). Its magnetic potential is therefore
\[
        B=\partial_x A_2-\partial_y A_1=4r^2,
\]
 so \(B\to +\infty\) as \(r\to\infty\).

\section{Additional remarks on unique continuation and compact resolvent}\label{sec:auxiliary}
The next remark records some settings in which the weak unique-continuation property holds, so that Remark~\ref{rem:localized-density} applies. It also emphasises that  the  stability of the unique continuation under  compactly supported zeroth-order perturbations is a feature of the
second-order elliptic operators and should not be assumed for non-elliptic
sums of squares.

\begin{remark}
Let \(\omega\subset X^\circ\) be open, and assume that
\[
  \omega\cap X_\alpha\neq \varnothing
  \qquad\text{for every connected component }X_\alpha\text{ of }X^\circ .
  \tag{\(C_\omega\)}
\]
When \(X^\circ\) is connected, this simply means that \(\omega\neq\varnothing\). Under condition \((C_\omega)\), the  property \((\mathrm{UCP}_\omega)\) holds for the second-order elliptic examples considered in
Examples~\ref{ex:Dir-Lapl}, \ref{ex:Neu-Lapl}, and \ref{ex:robin-arbitrary}. This is a consequence of  Aronszajn's
unique-continuation theorem; see~\cite[Remark~3, p.~248]{Aronszajn1957}.
The boundary condition plays no role in this interior argument. The same conclusion holds for the harmonic/anharmonic operators of
Example~\ref{ex:anharmonic} when  \(q(D)\) is second-order
elliptic. It also holds for the magnetic Schr\"odinger operators in
Examples~\ref{ex:magnetic-non-compact},\ref{ex:magnet.Sch.Neum}, and \ref{ex:magnetic-Dir}, after viewing the complex equation as a real
\(2\times 2\) elliptic system; see the system version~\cite[Remark~3, p.~248]{Aronszajn1957} of the unique-continuation theorem. In all these cases, adding a potential in
\(L^\infty_{\mathrm{loc}}\) only changes the lower-order terms, and so it  does not
affect the unique-continuation argument. By contrast, for the non-elliptic examples of Section~\ref{subsec:non-elliptic} one should not
infer \((\mathrm{UCP}_\omega)\) unless an independent weak
unique-continuation theorem is available for the particular operator under
consideration; Bahouri~\cite[(1.13), p.~140]{Bahouri1986} shows
that for
\[
  P_0=\partial_{x_1}^2+
  \bigl(\partial_{x_2}+x_1\partial_{x_3}\bigr)^2
\]
there exist a neighbourhood \(V\), functions \(a,u\in C^\infty(V)\), with
\(u\not\equiv 0\), such that
\[
  (P_0+a)u=0 \quad \text{in } V,
\]
and \(u\) vanishes on a non-empty open subset of \(V\). 
\end{remark}

\begin{proposition}
\label{prop:form-domain-compact-resolvent}
Let \(L\) be a self-adjoint operator on a Hilbert space \(\mathcal H\),
and assume that
\(
L\geq \gamma
\)
for some \(\gamma\in\mathbb R\). Let
\[
\mathcal Q(L)
:=
D\bigl((L+s)^{1/2}\bigr),
\qquad s>-\gamma,
\]
be the form domain of \(L\), endowed with its form norm. Then \(L\) has
compact resolvent if and only if the canonical embedding
\[
\mathcal Q(L)\hookrightarrow \mathcal H
\]
is compact.
\end{proposition}

\begin{proof}
Follows from  \cite[Theorem XIII.64, p. 245]{RS4}; see also \cite[Corollary 1.5]{LSW10}.
\end{proof}

\begin{proposition}
\label{prop:compact-resolvent-bounded-perturbations}
Let \(X\) be a smooth manifold, possibly non-compact, endowed with a positive
smooth density, and set
\(
\mathcal H:=L^2(X).
\)
Let \(\mathcal{Q}\) be an admissible perturbation space, and let
\[
L:\operatorname{Dom}(L)\subset \mathcal H\longrightarrow \mathcal H
\]
be self-adjoint and semibounded. The following assertions are equivalent:
\begin{enumerate}
    \item $L$ has compact resolvent;
    \item $L_{q_0}=L+M_{q_0}$ has compact resolvent for some $q_0\in \mathcal{Q}$;
    \item $L_q=L+M_q$ has compact resolvent for every  $q\in \mathcal{Q}$;
    \item The canonical embedding
\(
\mathcal Q(L)\hookrightarrow \mathcal H
\)
is compact, where $\mathcal Q(L)$ is the form domain of $L$. 
\end{enumerate}
Assume moreover that
\(
C_c^\infty(X^\circ,\mathbb R)\subset \mathcal Q.
\) If one, hence all, of the above equivalent conditions 
holds, then the set $\mathcal G_{\mathcal Q}$ is residual in \(\mathcal Q\).
\end{proposition}

\begin{proof}
Assume first that \(L\) has compact resolvent. Let \(q\in \mathcal Q\). Since
\(
q\in \mathcal Q\subset L^\infty(X,\mathbb R),
\)
the  operator \(M_q\) is bounded and self-adjoint on
\(\mathcal H\). Hence, by Lemma~\ref{lem:bounded-perturbations}, $L_q$ is also 
self-adjoint, bounded from below, and has compact resolvent. Thus \textit{1.} $\Rightarrow$ \textit{3.}, and \textit{3.} $\Rightarrow$ \textit{2.} is immediate. Conversely, assume that \textit{2.} holds.  Since \(L\) is semibounded and $M_{q_0}$ is 
bounded, \(L_{q_0}\) is also bounded from below. Moreover, we can write 
\(
L=L_{q_0}-M_{q_0},
\)
and since \(-M_{q_0}\) is bounded and self-adjoint on \(\mathcal H\), another
application of Lemma~\ref{lem:bounded-perturbations} shows that \(L\) has compact resolvent. Therefore \textit{2.} $\Rightarrow$ \textit{1.} The equivalence between \textit{1.} and \textit{4.} follows from Proposition \ref{prop:form-domain-compact-resolvent}. Finally, under the extra assumption that $C_c^\infty(X^\circ,\mathbb R)\subset \mathcal Q$, the  conclusion follows from Corollary~\ref{cor:automatic-splitting} and Theorem~\ref{thm:abstract}.
\end{proof}

\medskip

We point out that the compact-resolvent assumption of the abstract Theorem \ref{thm:abstract} is not essential
for the generic-simplicity mechanism itself; if one restricts attention to a
spectral region where the spectrum is discrete, consists of isolated
eigenvalues of finite multiplicity, and is stable under the admissible
perturbations, the same Baire-category argument applies. The following example on the spectrum of the Coulomb Hamiltonian, see e.g. \cite{Teschl2009}, illustrates this point.

\begin{example}
Let \(\gamma>0\), and let \(L_\gamma\) be the Friedrichs extension on
\(L^2(\mathbb R^3)\) associated with
\[
        -\Delta-\frac{\gamma}{|x|}.
\]
Let
\(
        \mathcal{Q} :=
        \left\{
        q\in C_b^\infty(\mathbb R^3,\mathbb R):
        q(x)\to0 \text{ as } |x|\to\infty
        \right\},
\)
endowed with the topology induced by \(C_b^\infty(\mathbb R^3,\mathbb R)\).
Then \(\mathcal{Q}\) is an admissible perturbation space and
\(
        C_c^\infty(\mathbb R^3,\mathbb R)\subset \mathcal{Q}.
\)
For \(q\in \mathcal{Q}\), set
\(
        L_{\gamma,q}:=L_\gamma+M_q .
\)
Then the set
\[
        G_\gamma :=
        \left\{
        q\in \mathcal Q:
        \text{ every eigenvalue of }L_{\gamma,q}
        \text{ in }(-\infty,0)
        \text{ is simple}
        \right\}
\]
is residual in \(\mathcal{Q}\). 

Indeed, for \(q\in \mathcal Q\), write
\(
        L_{\gamma,q}
        =
        -\Delta+a_{\gamma,q}(x),\)
         where \(
        a_{\gamma,q}(x):=
        -\frac{\gamma}{|x|}+q(x).
\) 
Since $ a_{\gamma,q}(x)\to0$ as $|x| \to \infty$ and $        a_{\gamma,q}\in L^2_{\mathrm{loc}}(\mathbb R^3),$  by Persson's
 \cite[Theorem~2.1, Proposition 3.1]{Persson1960},  the bottom of the essential spectrum is $0$, and hence for every $\alpha<0$ the spectral part of \(L_{\gamma,q}\) in $(-\infty,\alpha]$ is finite and consists of isolated eigenvalues of finite multiplicity.
In particular, for every \(N\geq1\), the interval
\(
        (-\infty,-1/N]:=(-\infty, \alpha_N]
\)
contains only finitely many eigenvalues of \(L_{\gamma,q}\), counted with
multiplicity. Define
\[
        G_{\gamma,N}
        :=
        \left\{
        q\in \mathcal Q:
        \text{ every eigenvalue of }L_{\gamma,q}
        \text{ in }(-\infty,\alpha_N]
        \text{ is simple}
        \right\}.
\]
Then
\[
        G_\gamma=\bigcap_{N\geq1}G_{\gamma,N},
\]
and it remains to prove that each \(G_{\gamma,N}\) is open and dense in \(\mathcal Q\). We first prove openness.  Fix \(p\in G_{\gamma,N}\).  Since the spectrum
below \(0\) is discrete, we may choose
\[
        \beta\in(\alpha_N,0)
        \quad\text{such that}\quad
        \sigma(L_{\gamma,p})\cap(\alpha_N,\beta]=\varnothing .
\]
The spectral part of \(L_{\gamma,p}\) below \(\beta\) is finite, and the
eigenvalues in \((-\infty,\alpha_N]\) are simple.  By the same min--max
estimate as in Lemma~\ref{lem:lipschitz}, applied to this finite spectral part below
\(\beta\), these eigenvalues remain separated for all \(q\) sufficiently
close to \(p\) in \(L^\infty(\mathbb R^3)\).  Moreover, by choosing \(q\)
sufficiently close to \(p\), no spectral point lying above \(\beta\) can
enter the interval \((-\infty,\alpha_N]\).  Hence
\(
        q\in G_{\gamma,N}.
\)
Since the inclusion
\(
        Q\hookrightarrow L^\infty(\mathbb R^3,\mathbb R)
\)
is continuous, \(G_{\gamma,N}\) is open in \(\mathcal Q\).

We now prove density.  Fix \(p\in \mathcal Q\), let \(U\) be a neighborhood of
\(p\) in \(\mathcal Q\), and choose
\[
        \beta\in(\alpha_N,0)
        \quad\text{such that}\quad
        \beta\notin\sigma(L_{\gamma,p}).
\]
The spectral part of \(L_{\gamma,p}\) in \((-\infty,\beta)\) is finite. Since
\(
        C_c^\infty(\mathbb R^3,\mathbb R)\subset \mathcal Q,
\)
Corollary~\ref{cor:automatic-splitting} gives Hypothesis~{\rm (H)} for every multiple eigenspace of
\(L_{\gamma,p}\).  We can therefore repeat the argument of
Lemmas~\ref{lem:reduce_mult_abstract} and~\ref{lem:dense} on the finite spectral part below \(\beta\).  The
only change is that Proposition~\ref{lem:kato} is used only in its local form: if
\(\lambda<\beta\) is an isolated eigenvalue of finite multiplicity, then
the eigenvalues near \(\lambda\) arising from \(\lambda\) admit real-analytic
branches under  bounded potential perturbations; see
Kato~\cite[Chapter VII, \S3.2, p.~387]{Kato}.

Thus, by the same  argument as in Lemma~\ref{lem:reduce_mult_abstract}, the multiplicity
of a multiple eigenvalue below \(\beta\) can be reduced by an arbitrarily
small perturbation in \(\mathcal Q\).  Since there are only finitely many
eigenvalues below \(\beta\), repeating this finitely many times gives
some
\(
        q\in U
\)
such that every eigenvalue of \(L_{\gamma,q}\) below \(\beta\) is simple.
Since
\(
        (-\infty,\alpha_N]\subset(-\infty,\beta),
\)
we have
\(
        q\in U\cap G_{\gamma,N},
\)
and so  \(G_{\gamma,N}\) is dense in \(\mathcal Q\).

Summarising, \(G_{\gamma,N}\) is open and dense in \(\mathcal{Q}\) for every \(N\geq1\).
Since \(\mathcal Q\) is a Fr\'echet space, it is a Baire space.  Hence
\[
        G_\gamma
        =
        \bigcap_{N\geq1}G_{\gamma,N}
\]
is residual in \({\mathcal Q}\).
\end{example}

The Coulomb term is only a model case: it may be replaced by any real-valued \(V\in L^2(\mathbb R^3)+(L^\infty(\mathbb R^3))_{\varepsilon}\), i.e. by a potential for which there exist \(V_n\in L^2(\mathbb R^3)\) with \(V-V_n\in L^\infty(\mathbb R^3)\) and \(\|V-V_n\|_{\infty}\to0\).
For such \(V\), the operator \(V(-\Delta+1)^{-1}\) is compact, see e.g. \cite[Example~6, p. 117]{RS4}, hence \(\sigma_{\mathrm{ess}}(L_V)=[0,\infty)\), where \(L_V\) is the self-adjoint realization of \(-\Delta+V\); setting \(L_{V,q}:=L_V+M_q\), \(q\in \mathcal Q\), the same conclusion holds.

\section*{Acknowledgements}
\addcontentsline{toc}{section}{Acknowledgements}

We thank Luc Hillairet for helpful discussions, for his feedback on this work, and for communicating
the paper \cite{HarakehHillairet}. We also thank Irene
Silvestre Rosello for preliminary discussions on the subject, and Pier Domenico Lamberti for communicating the paper \cite{LZ}.

\section*{Data availability}
No data was used for the research described in the article.

\end{document}